\documentclass{amsart}
\usepackage{amsthm}

\newtheorem{thm}{Theorem}[section]
\newtheorem{cor}[thm]{Corollary}
\newtheorem{lem}[thm]{Lemma}

\newtheorem{prop}[thm]{Proposition}

\newtheorem{rem}[thm]{Remark}
\newtheorem{ex}[thm]{Example}

\newtheorem{defn}[thm]{Definition}

\numberwithin{equation}{section}

\newcommand{\trace}{\text{trace}}

\begin{document}

\title[Viscosity solutions to curvature flow equations on manifolds]{
Generalized motion of level sets by functions of their curvatures on
Riemannian manifolds}
\author{D. Azagra, M. Jim{\'e}nez-Sevilla, F. Maci{\`a}}

\address{Departamentos de An{\'a}lisis Matem{\'a}tico y de Matem{\'a}tica Aplicada\\ Facultad de
Matem{\'a}ticas\\ Universidad Complutense\\ 28040 Madrid, Spain}

\date{August 29, 2007}

\thanks{D. Azagra was supported by grants MTM-2006-03531 and UCM-CAM-910626.
M. Jimenez-Sevilla was supported by a fellowship of the Ministerio
de Educacion y Ciencia, Spain. F. Maci{\`a} was supported by program
"Juan de la Cierva" and projects MAT2005-05730-C02-02 of MEC
(Spain) and PR27/05-13939 UCM-BSCH (Spain).}

\email{azagra@mat.ucm.es, marjim@mat.ucm.es,
fabricio\_macia@mat.ucm.es}



\begin{abstract}
We consider the generalized evolution of compact level sets by
functions of their normal vectors and second fundamental forms on
a Riemannian manifold $M$. The level sets of a function
$u:M\to\mathbb{R}$ evolve in such a way whenever $u$ solves an
equation $u_{t}+F(Du, D^{2}u)=0$, for some real function $F$
satisfying a geometric condition. We show existence and uniqueness
of viscosity solutions to this equation under the assumptions that
$M$ has nonnegative curvature, $F$ is continuous off $\{ Du=0\}$,
(degenerate) elliptic, and locally invariant by parallel
translation. We then prove that this approach is geometrically
consistent, hence it allows to define a generalized evolution of
level sets by very general, singular functions of their
curvatures. For instance, these assumptions on $F$ are satisfied
when $F$ is given by the evolutions of level sets by their mean
curvature (even in arbitrary codimension) or by their positive
Gaussian curvature. We also prove that the generalized evolution
is consistent with the classical motion by the corresponding
function of the curvature, whenever the latter exists. When $M$ is
not of nonnegative curvature, the same results hold if one
additionally requires that $F$ is uniformly continuous with
respect to $D^2 u$. Finally we give some counterexamples showing
that several well known properties of the evolutions in
$\mathbb{R}^{n}$ are no longer true when $M$ has negative
sectional curvature.
\end{abstract}

\maketitle

\section{Introduction}

In the last 30 years there has been a lot of interest in the
evolution of hypersurfaces of $\mathbb{R}^{n}$ by functions of
their curvatures. In this kind of problem one is asked to find a one
parameter family of orientable, compact hypersurfaces $\Gamma_{t}$
which are boundaries of open sets $U_{t}$ and satisfy
\begin{eqnarray}\label{geometric equation}
& & V=-G(\nu, D\nu) \textrm{ for } t>0, \, x\in \Gamma_{t},  \,
\textrm{
and } \\
& & \Gamma_{t}|_{t=0}=\Gamma_{0} \notag
\end{eqnarray}
for some initial set $\Gamma_{0}=\partial U_{0}$, where $v$ is the
normal velocity of $\Gamma_{t}$, $\nu=\nu(t,\cdot)$ is a normal
field to $\Gamma_t$ at each $x$, and $G$ is a given (nonlinear)
function.

Two of the most studied examples are the evolutions by mean
curvature and by (positive) Gaussian curvature. In both cases,
short time existence of classical solutions has been established.
 For strictly convex initial data $U_{0}$, it has been shown that
$U_{t}$ shrinks to a point in finite time, and moreover,
$\Gamma_{t}$ becomes spherical at the end of the contraction. See
\cite{Andrews, GageHamilton, Grayson2, Grayson3, Huisken1,
Huisken2, Tso} and the references therein.

For dimension $n\geq 3$ it has been shown \cite{Grayson} that a
hypersurface evolution $\Gamma_{t}$ may develop singularities
before it disappears. Hence it is natural to try to develop weak
notions of solutions to \eqref{geometric equation} which allow to
deal with singularities of the evolutions, and even with nonsmooth
initial data $\Gamma_{0}$.

There are two mainstream approaches concerning weak solutions of
\eqref{geometric equation}: the first one uses geometric measure
theory to construct (generally nonunique) varifold solutions, see
\cite{Brakke, Ilmanen2}, while the second one adapts the theory of
second order viscosity solutions developed in the 1980's (see
\cite{CIL} and the references therein) to show existence and
uniqueness of level-set weak solutions to \eqref{geometric
equation}.

In this paper we will focus on this second approach. The first
works to develop a notion of {\em viscosity} level set solution to
\eqref{geometric equation} were those of Evans and Spruck
\cite{ES1} and, independently developed, Chen, Giga and Goto
\cite{ChGG}. This was followed by many important developments,
which we find impossible to properly quote here; we refer the
reader to the very comprehensive monograph \cite{Giga} and the
bibliography therein. This level set approach consists in
observing that a smooth function
$u:[0,T]\times\mathbb{R}^{n}\to\mathbb{R}$ with $Du:=D_{x}u\neq 0$
has the property that all its level sets evolve by
\eqref{geometric equation} if and only if $u$ is a solution of
\begin{equation}\label{viscosity equation}
u_{t}+F(Du, D^{2}u)=0,
\end{equation}
where $F$ is related to $G$ in \eqref{geometric equation} through
of the following formula:
    \begin{equation}\label{relation between F, G}
    F(p, A)=|p|\, G\left(\frac{p}{|p|}, \,
    \frac{1}{|p|}\left(I-\frac{p\otimes p}{|p|^{2}}\right)A\right).
    \end{equation}
The function $F$ is assumed to be continuous off $\{p=0\}$ and
(degenerate) elliptic, that is
\begin{equation}
F(p,B)\leq F(p, A) \, \textrm{ whenever }\, A\leq B.
\end{equation}
Because of \eqref{relation between F, G}, $F$ also has the following
geometric property:
\begin{equation}\label{geometricity}
F(\lambda p, \lambda A+\mu p\otimes p)=\lambda F(p, A) \textrm{ for
all } \lambda>0, \mu\in\mathbb{R}.
\end{equation}
The function $F$ does not generally admit any continuous extension
to $\mathbb{R}^{n}\times\mathbb{R}^{n^2}$ but, if it is bounded
near $\{p=0\}$ (this is the case of the mean curvature evolution
equation), one can show that there is a unique viscosity solution
to \eqref{viscosity equation} with initial datum $u(0,x)=g(x)$
(for any continuous $g$ such that $\Gamma_{0}=\{x: g(x)=0\}$).
Next one can also see that if $\theta:\mathbb{R}\to\mathbb{R}$ is
continuous and $u$ is a solution of \eqref{viscosity equation}
then $\theta\circ u$ is a solution too, and this, together with a
comparison principle, allows to show that the generalized
geometric evolution
$$\Gamma_{0}\to \Gamma_{t}:=\{x: u(t,x)=0\}$$
is well defined (that is, the zero level set of a solution to
\eqref{viscosity equation} only depends on the zero level set of
its initial datum). It is also possible to show that this level
set evolution agrees with any classical solution of
\eqref{geometric equation}.

When $F(p, A)$ is not bounded as $p\to 0$ (this is the case of
more singular equations such as the Gaussian evolution), then the
standard notion of viscosity solution to \eqref{viscosity
equation} (as is used, for instance, in \cite{ChGG}) at points
$z=(t,x)$ where the test function $\varphi$ satisfies
$D\varphi(z)=0$ is not suitable to tackle the problem. In this
case two different modifications of the notion of solution have
been proposed in the literature.

One possibility is simply not to specify any condition for the
derivatives of a test function $\varphi$ such that $u-\varphi$
attains a maximum or a minimum at a point $(t_{0}, x_{0})$ with
$D\varphi(t_{0}, x_{0})=0$. This is Goto's approach in
\cite{Goto}. When one uses this definition of solution, the
corresponding comparison theorem becomes harder to prove, and it
is indeed a stronger statement since the class of solutions
becomes bigger in this case, while the existence result is
comparatively weaker.

The other possibility is to make the class of test functions
$\varphi$ smaller, in a clever way so that, if $z_k\to z_0$ and
$D\varphi_k(z_k)\to 0$, one can show that $F(D\varphi_k(z_k),
D^{2}\varphi_k(z_k))$ goes to $0$, and then to demand that a
subsolution $u$ should satisfy that if $u-\varphi$ has a maximum
at $z_{0}$ then $\varphi_{t}(z_{0})\leq 0$. This is what Ishii and
Souganidis did in \cite{IS}. The corresponding (sub)solutions are
called $\mathcal{F}$-(sub)solutions in Giga's book \cite{Giga}. In
this approach the maximum principle is relatively easier to prove,
while existence becomes harder (and is really a stronger result,
because the class of solutions is smaller in this case).

\medskip

The aim of this paper is to investigate to what extent one can
develop a general theory of (viscosity) level-set solutions to the
problem of the evolution of hypersurfaces by functions of their
curvatures in a Riemannian manifold. To the best of our knowledge,
the only work in this direction is Ilmanen's paper \cite{Ilmanen}
(in fact this is the only paper we know of in which second order
viscosity solutions are employed to deal with a second order
evolution equation within the context of Riemannian manifolds). In
\cite{Ilmanen} Ilmanen shows existence and uniqueness of a
(standard) viscosity solution to the mean curvature evolution
equation, that is \eqref{geometric equation} in the case when $F$
is given by $$F(p, A)=-\trace\left(\left(I-\frac{p\otimes
p}{|p|^{2}}\right)A\right),$$ with initial condition
$u(0,x)=g(x)$, thus obtaining a corresponding generalized
evolution by mean curvature, some of whose geometric properties he
next studies. For instance, he proves that if noncompact initial
data $\Gamma_{0}$ are allowed then one loses uniqueness of the
generalized geometric evolution.

In recent years, an interest has grown in the use of viscosity
solutions of (first order) Hamilton-Jacobi equations defined on
Riemannian manifolds (in relation to dynamical systems, to
geometric problems, or from a theoretical point of view), see
\cite{ManteMenu, FathiSino1, FathiSino2,  AFL1, LedZhu,
DaviniSino, GurskyViaclovsky}, but no second order theory, apart
from Ilmanen's paper, has apparently been developed for parabolic
equations (in the case of stationary, degenerate elliptic
equations, such a study was recently started in \cite{AFS}).

We believe that a level set method for generalized evolution of
hypersurfaces by functions of their curvatures can be useful in
the setting of Riemannian manifolds. On the one hand we think that
it is very natural, from a geometric point of view, to try to
study the evolutions of level sets in a general Riemannian
manifold $M$ by their Gaussian (or by other functions of their)
curvatures, in a way that is supple enough so that nonsmooth
initial data and singularities of the evolutions are allowed. On
the other hand, as one sees, for instance, by restricting to the
case $M=\mathbb{R}^{n}$ endowed with a non Euclidean metric, the
tools developed here allow to treat level set evolutions in
inhomogeneous media, in which the function $F$ depends (in a very
special manner) on the position variable $x$.

\medskip

Let us briefly describe the main results of this paper. In Section
2 we consider equations of the form \eqref{geometric equation},
\eqref{viscosity equation} for level sets of functions $u$ defined
on a Riemannian manifold $M$, and we show how the $F$'s
corresponding to the evolutions by mean curvature (even in
arbitrary codimension, in the line of \cite{AS}) and by (positive)
Gaussian curvature are extended to $J_{0}^{2}(M)$ in such a way
that $F$ is (degenerate) {\em elliptic, translation invariant,
geometric}, and continuous off $\{Du=0\}$ (see properties ({\bf A
- D}) in Section 2 below). Following \cite{IS, Giga}, for each $F$
we next define an appropriate class of test functions
$\mathcal{A}(F)$ which allows us to deal with equation
\eqref{viscosity equation} on $M$, and we define the corresponding
class of $\mathcal{F}$-solutions, see Definitions \ref{def of the
F class}, \ref{def of the A(F) class} below. We also show that for
all $F$ which are continuous off $\{Du=0\}$, elliptic, translation
invariant and geometric, one has that
$\mathcal{A}(F)\neq\emptyset$ provided that $M$ is compact.
Moreover, in the cases when $F$ is given by the mean curvature or
the Gaussian curvature evolution equations, we have
$\mathcal{A}(F)\neq\emptyset$ no matter whether $M$ is compact or
not.

In Section 3 we present some technical results that will be used
later on in the proofs of the main results.

Section 4 is devoted to proving a comparison result for viscosity
solutions of \eqref{viscosity equation} on $M$: under the above
assumptions on $F$ (namely, continuity, ellipticity, geometricity
and translation invariance) we show that if $M$ has nonnegative
curvature $u$ is a subsolution, $v$ is a supersolution, $u\leq v$
on $\{0\}\times M$, and $\limsup_{(t,x)\to\infty}(u-v)\leq 0$
(this condition is understood to be requiring nothing when $M$ is
compact), then $u\leq v$ on $[0,T]\times M$. When $M$ is not of
nonnegative curvature, we have to additionally require that $F$ be
uniformly continuous with respect to $D^2 u$.

In Section 5 we show that Perron's method (first used in \cite{I})
works to produce $\mathcal{A}(F)$-solutions of \eqref{viscosity
equation} on a Riemannian manifold $M$, provided that comparison
holds and $\mathcal{A}(F)\neq\emptyset$.

Therefore, for all such $M$ and $F$, for every compact subset
$\Gamma_{0}$ of $M$, and for every continuous function $g$ on $M$
such that $\Gamma_{0}=\{x\in M: g(x)=0\}$, there exists a unique
solution of \eqref{viscosity equation} on $M$ with initial
condition $u(0,\cdot)=g$. One can then define, for each compact
$\Gamma_{0}$, an evolution $\Gamma_{t}=\{x\in M: u(t,x)=0\}$,
$t\geq 0$. In Section 6 we see that $\Gamma_{t}$ does not depend
on the function $g$ chosen to represent $\Gamma_{0}$, and
consequently the generalized geometric evolution
$\Gamma_{0}\mapsto\Gamma_{t}$ is well defined.

Next, in Section 7 we prove that this generalized evolution is
consistent with the classical motion, whenever the latter exists.
Namely, if $\left(  \Gamma_{t}\right)  _{t\in\left[ 0,T\right] }$
is a family of smooth, compact, orientable hypersurfaces in a
Riemannian manifold $M$ evolving according to a classical
geometric motion, locally depending only on its normal vector
fields and second fundamental forms according to an equation of
the form \eqref{geometric equation}, and $\Gamma_{0}$ can be
represented as the zero level set of a smooth function $g$ on $M$,
then $\Gamma_{t}$ coincides with the generalized level set
evolution (with initial datum $\Gamma_{0}$) defined above.

Finally, in Section 8, we give some counterexamples showing that
several well known properties of generalized solutions to the mean
curvature flow cannot be extended from Euclidean spaces to
Riemannian manifolds of negative sectional curvature. For
instance, Ambrosio and Soner \cite{AS} showed that the distance
function from $\Gamma_{t}\subset\mathbb{R}^{n}$ given by
$|d|(t,x)=\textrm{dist}(x, \Gamma_{t})$ is a supersolution of
\eqref{viscosity equation} when $F$ corresponds to the mean
curvature evolution equation. We show that this result fails when
$\mathbb{R}^{n}$ is replaced with a manifold of negative sectional
curvature. On the other hand, if $M$ has negative curvature, then
equation \eqref{viscosity equation} does not preserve Lipschitz
properties of the initial data, in contrast with \cite[Chapter
3]{Giga}. And, again in the case of the mean curvature flow, if
$\Gamma_{0}$, $\hat{\Gamma}_{0}$ are smooth $1$-codimensional
submanifolds of a manifold $M$ of negative curvature, then the
function $t\mapsto\textrm{dist}(\Gamma_{t}, \hat{\Gamma}_{t})$ can
be decreasing, in contrast with \cite[Theorem 7.3]{ES1}.

An the end of this article, the reader will find an appendix
describing a comparison and an existence result for (standard)
viscosity solutions to general evolution equations of the form
$$
u_{t}+F(x,t, u, Du, D^2 u)=0
$$
where $F$ has no singularities. We omit the proofs because they
resemble (and are easier than) those of the main comparison and
existence result for $\mathcal{F}$-solutions of \eqref{viscosity
equation} given in Sections 4 and 5.

\medskip

\noindent {\bf Notation.} $M$ will always be a finite-dimensional
Riemannian manifold. We will write
$\left\langle\cdot,\cdot\right\rangle$ for the Riemannian metric
and $|\cdot|$ for the Riemannian norm on $M$. The tangent and
cotangent space of $M$ at a point $x$ will be respectively denoted
by $TM_{x}$ and $TM_{x}^{*}$. We will often identify them via the
isomorphism induced by the Riemannian metric. The space of
bilinear forms on $TM_{x}$ (respectively symmetric bilinear forms)
will be denoted by $\mathcal{L}^{2}(TM_{x})$ (resp.
$\mathcal{L}^{2}_{s}(TM_{x})$). Elements of
$\mathcal{L}^{2}(TM_{x})$ will be denoted by the letters $A, B, P,
Q$, and those of $TM_{x}^{*}$ by $\zeta, \eta$,
 etc. Also, we will respectively denote the cotangent bundle and the tensor bundle
of symmetric bilinear forms in $M$ by
\[
TM^{*}:=\bigcup_{x\in M}TM_{x}^{*},\qquad T_{2,
s}(M):=\bigcup_{x\in M}\mathcal{L}^{2}_{s}(TM_x).
\]
We will also consider the two-jet bundles:
\[
J^{2}M:=\bigcup_{x\in
M}TM_{x}^{*}\times\mathcal{L}^{2}_{s}(TM_x),\qquad
J_{0}^{2}(M):=\bigcup_{x\in
M}\left(TM_{x}^{*}\setminus\left\{0_{x}\right\}\right)\times\mathcal{L}^{2}_{s}(TM_x).
\]
The letters $X, Y, Z$ will stand for smooth vector fields on $M$,
and $\nabla_{Y}X$ will always denote the covariant derivative of
$X$ along $Y$. Curves and geodesics in $M$ will be denoted by
$\gamma$, $\sigma$, and their velocity fields by $\gamma',
\sigma'$. If $X$ is a vector field along $\gamma$ we will often
denote $X'(t)= \frac{d}{dt}X(t)=\nabla_{\gamma'(t)}X(t)$. Recall
that $X$ is said to be parallel along $\gamma$ if $X'(t)=0$ for
all $t$. The Riemannian distance in $M$ will always be denoted by
$d(x,y)$ (defined as the infimum of the lengths of all curves
joining $x$ to $y$ in $M$).

Given a smooth function $u:M\rightarrow\mathbb{R}$, we will denote
its differential by $D_{x}u \in TM^{*}$; its gradient vector field
will be written as $\nabla u$, and its Hessian as $D_{x}^{2}u$.
Recall that, for any two vector fields $X,Y$ satisfying $X(p)=v,
Y(p)=w$ at some $p \in M$ we have:
\[
D_{x}^{2}u\left(X,Y\right):=\left\langle\nabla_{Y}\nabla
u,X\right\rangle, \qquad
D_{x}^{2}u\left(v,w\right):=D_{x}^{2}u\left(X,Y\right)(p).
\]
Given a function $v:M\rightarrow\mathbb{R}$ we will use the
notation:
\[
\begin{array}{c}
 v^*(t,x)=\lim_{r\downarrow 0}\sup\{v(s,y):y\in M,\, s>0,\, |t-s|\leq r, \, d(y,x)\leq r\},\medskip
 \\
 v_*(t,x)=\lim_{r\downarrow 0}\inf\{v(s,y):y\in M,\, s>0,\, |t-s|\leq r, \, d(y,x)\leq r\};  \\
\end{array}
\]
that is $v^*$ denotes the upper semicontinuous envelope of $v$
(the smallest upper semicontinuous function, with values in
$[-\infty,\infty]$, satisfying $v\leq v^*$), and similarly $v_*$
stands for the lower semicontinuous envelope of $v$.

We will make frequent use of the exponential mapping $\exp_{x}$
and of the parallel translation along a geodesic $\gamma$. Recall
that for every $x\in M$ there exists a mapping $\exp_{x}$, defined
on a neighborhood of $0$ in the tangent space $TM_x$, and taking
values in $M$, which is a local diffeomorphism and maps straight
line segments passing through $0$ onto geodesic segments in $M$
passing through $x$. The exponential mapping induces a local
diffeomorphism on the cotangent space $TM_{x}^{*}$, via the
identification given by the metric, that will be also denoted by
$\exp_{x}$. On the other hand, for a minimizing geodesic
$\gamma:[0, \ell]\to M$ connecting $x$ to $y$ in $M$, and for a
vector $v\in TM_{x}$ there is a unique parallel vector field $P$
along $\gamma$ such that $P(0)=v$, this is called the parallel
translation of $v$ along $\gamma$. The mapping $TM_{x}\ni v\mapsto
P(\ell)\in TM_{y}$ is a linear isometry from $TM_{x}$ onto
$TM_{y}$ which we will denote by $L_{xy}$. This isometry naturally
induces an isometry between the space of bilinear forms on
$TM_{x}$ and the space of bilinear forms on $TM_{y}$. Whenever we
use the notation $L_{xy}$ we assume implicitly that $x$ and $y$
are close enough to each other so that this makes sense.

By $i_{M}(x)$ we will denote the injectivity radius of $M$ at $x$,
that is the supremum of the radius $r$ of all balls $B(0_{x}, r)$
in $TM_{x}$ for which $\exp_{x}$ is a diffeomorphism from
$B(0_{x}, r)$ onto $B(x,r)$. Similarly, $i(M)$ will denote the
global injectivity radius of $M$, that is $i(M)=\inf\{i_{M}(x) :
x\in M\}$. Recall that the function $x\mapsto i_{M}(x)$ is
continuous. In particular, if $M$ is compact, we always have
$i(M)>0$.

\bigskip

\section[general curvature evolution equations]{General curvature evolution
equations on Riemannian manifolds}

Consider the following evolution equation on a
 Riemannian manifold $M$, given by
    \begin{equation} \label{curvatureevolutioneq} \tag{\texttt{CEE}}
   u_t-F(Du, D^2 u) \ =0 \,\, \textrm{ on }
    \,\ (0,T)\times M,
    \end{equation}
    $$ u(0,x)=g(x), \ \   \textrm{ on  } x\in M, $$
where $u$ is a function of $(t,x)\in [0,T)\times M$.

In what follows, $u_t$, $Du$ and $D^2 u$ will stand for $D_{t}u$,
$D_x u(t,x)\in TM^*_x$ and $D_x^2u(t,x)\in \mathcal L^2_s(TM_x)$,
respectively. The function $F$ is assumed to be continuous on the
normal vector to the level set $\Gamma_t =\{x\in M: u(t,x)=0\}$
and on the curvature tensor, and having the form
    \begin{equation} \label{FandG}
    F(\zeta, A)=|\zeta|\, G\left(\frac{\zeta}{|\zeta|}, \,
    \frac{1}{|\zeta|}\left(I-\frac{\zeta\otimes\zeta}{|\zeta|^{2}}\right)A\right),
    \end{equation}
for all $\zeta\in TM_{x}^{*}\setminus\{0_x\}$ and $A\in \mathcal
L^2_s(TM_x)$, where $G$ is any (nonlinear) function such that:
\begin{itemize}
\item[(\bf{A})] $F:J_{0}^{2}(M)\rightarrow \mathbb R$ is continuous;
\item[(\bf{B})] $F$ is {\em (degenerate) elliptic}, that is,  $$A\le B \implies F(\zeta,B)\le
F(\zeta, A),$$ for all $x\in M,$  \, $\zeta\in TM^*_x
\setminus\{0\}, \, A, B\in\mathcal{L}^{2}_{s}(TM_{x})$;
\item[(\bf{C})] $F$ is {\em translation invariant}, meaning that
there exists $\tau >0$ such that $$
 F(L_{xy}\zeta, A)=F(\zeta,L_{yx}(A)),
$$
for every $x,y \in M,$ \, $d(x,y)< \tau$, \, $\zeta \in
TM^*_x\setminus \{0\}$, \, $A\in  \mathcal{L}^2_s(TM_y)$.
\end{itemize}
Notice that, because
$$\left(I-\frac{\zeta\otimes\zeta}{|\zeta|^{2}}\right)\, (\zeta\otimes\zeta) = 0, $$ any
function $F$ of the form \eqref{FandG} also satisfies
\begin{itemize}
\item[(\bf{D})] $F$ is {\em geometric}, that is, $$F(\lambda\zeta, \lambda A+\mu\zeta\otimes\zeta)=\lambda F(\zeta,
A)$$ for every $\lambda>0$, $\mu\in\mathbb{R}$.
\end{itemize}

   \medskip

Two very important problems where such functions $F$ arise are the
evolutions of level sets by mean curvature and by Gaussian
curvature.
\medskip

\begin{ex}\label{example mean curvature} {\em \bf Motion of level sets by their mean curvature.} \end{ex}
If $u$ is a function on $[0,T]\times M$ such that $Du(t,x)\neq 0$
for all $t, x$ with $u(t,x)=c$, then each level set
$\Gamma_t=\{u(t,\cdot)=c \}$ evolves according to its mean curvature
if and only if $u$ satisfies
$$\frac{u_t}{|Du|}=\textrm{div}\left(\frac{Du}{|Du|}\right)$$ (that is, the
normal velocity of $\Gamma_{t}$ at a point $x$ equals $(n-1)$ times
the mean curvature of $\Gamma_{t}$ at $x$), which in turn is
equivalent to
    \begin{equation} \label{meancurvature} \tag{\texttt{MCE}}
   u_t-\trace \left( \left( I-\frac{Du \otimes Du}{|Du|^2}\right) D^2u\right) =0 \,\, \textrm{ on }
    \,\ (0,T)\times M.
    \end{equation}
That is, $u_{t}+F(Du, D^2 u)=0$, where
\begin{equation} \label{meancurvature2}
F(\zeta, A)=-\textrm{trace} \left( \left( I-\frac{\zeta \otimes
\zeta}{|\zeta|^2}\right) A\right).
\end{equation}
It is not difficult to see that the function
$F:J_{0}^{2}(M)\longrightarrow \mathbb R$ is continuous (though
the function $F$ remains undefined at $\zeta=0$ and, in fact,
there is no continuous extension of $F$ to $J^{2}(M)$.
Nevertheless, $F(\zeta, A)$ remains bounded as $\zeta\to 0$).

Let us now check that the function $F$ is degenerate elliptic. If
$P\leq Q$, since $R:= \displaystyle{I- \frac{\zeta\otimes
\zeta}{|\zeta|^2}\geq 0}$ and $S:=Q-P\geq 0$, we obtain from the
properties of the trace that $\trace(RS)\geq 0$ and therefore
\begin{equation}
F(\zeta,P)-F(\zeta,Q) =\trace \left(\left(  I- \frac{\zeta\otimes
\zeta}{|\zeta|^2} \right) \left( Q-P\right)\right)\geq 0.
\end{equation}

Finally, let us see that the function $F$ in \eqref{meancurvature2}
is translation invariant. Notice that
$\trace(A)=\trace(L_{xy}^{-1}\circ A\circ L_{xy})=
\trace(L_{yx}(A))$, and  \begin{align*}\trace \left(
\frac{\zeta\otimes \zeta}{|\zeta|^2} \circ L_{yx}(A) \right)&=
\trace \left(  L_{xy} \circ \frac{\zeta\otimes \zeta}{|\zeta|^2}
\circ L_{yx}(A) \circ L_{xy}^{-1}\right) \\ & = \trace \left( L_{xy}
\circ \frac{\zeta\otimes \zeta}{|\zeta|^2} \circ L_{xy}^{-1} \circ A
\right).
\end{align*}
On the other hand,
\begin{equation} \label{MCEInvariant}
\displaystyle{ L_{xy} \circ \frac{\zeta\otimes \zeta}{|\zeta|^2}
\circ L_{xy}^{-1}= \frac{L_{xy}\zeta\otimes
L_{xy}\zeta}{|L_{xy}\zeta|^2}},
\end{equation}
hence we immediately deduce that $F(L_{xy}\zeta,
A)=F(\zeta,L_{yx}(A))$ whenever $d(x,y) < i(M)$, $\zeta\in TM_x,
A\in\mathcal{L}^{2}_{s}(TM_y)$.

\bigskip

\begin{ex} {\em \bf Motion of level sets by their Gaussian curvature.}
\end{ex}
Now, if $u$ is a function on $[0,T]\times M$ such that $Du(t,x)\neq
0$ for all $t, x$ with $u(t,x)=c$, then all level sets
$\Gamma_t=\{u(t,\cdot)=c \}$ evolve according to their Gaussian
curvature if and only if $u$ satisfies
$$\frac{u_t}{|Du|}=\textrm{det}\left(\nabla^{T}\left(\frac{\nabla u}{|\nabla
u|}\right)\right),$$ where $\nabla^{T}$ stands for the orthogonal
projection onto $T\Gamma_t$ of the covariant derivative in $M$.
This equation is equivalent to
    \begin{equation} \label{Gaussiancurvature} \tag{\texttt{GCE}}
   u_t-|Du|\textrm{det} \left( \frac{1}{|Du|}\left( I-\frac{Du \otimes Du}{|Du|^2}\right) D^2u +
\frac{Du \otimes Du}{|Du|^2}\right) =0.
    \end{equation}
That is, $u_{t}+H(Du, D^2 u)=0$, where $$H(\zeta,
A)=-|\zeta|\textrm{det} \left( \left( I-\frac{\zeta\otimes
\zeta}{|\zeta|^2}\right)A +\frac{\zeta \otimes
\zeta}{|\zeta|^2}\right).$$

However, the function $H$ is not elliptic, so this problem cannot
be treated, in its most general form, with the theory of viscosity
solutions. Nevertheless, if our initial data $u(0,x)=g(x)$
satisfies that $D^2 g(x)\geq 0$ (that is, if the initial
hypersurface $\Gamma_{0}=\{x\in M: g(x)=c\}$ has nonnegative
Gaussian curvature) then it is reasonable, and consistent with the
classical motion of convex surfaces by their Gaussian curvature,
to assume that $D^2 u(t,x)\geq 0$ for all $(t,x)$ with $u(t,x)=c$
(that is, $\Gamma_t$ will have nonnegative Gaussian curvature as
long as it exists). In this case our equation becomes
    \begin{equation} \label{positiveGaussiancurvature}
\tag{\texttt{+GCE}}
   u_t-|Du|\textrm{det}_{+} \left( \frac{1}{|Du|}\left( I-\frac{Du \otimes Du}{|Du|^2}\right) D^2u +
\frac{Du \otimes Du}{|Du|^2}\right) =0,
    \end{equation}
where $\textrm{det}_{+}$ is defined by
$$\textrm{det}_{+}(A)=\prod_{j=1}^{n} \max\{\lambda_{j},0\}$$
if $\lambda_{1}, ..., \lambda_{n}$ are the eigenvalues of $A$.
That is, $$u_t+F(Du, D^2u)=0,$$ where
\begin{equation} \label{F for Gaussian curvature}
F(\zeta, A)=-|\zeta|\textrm{det}_{+} \left(\frac{1}{|\zeta|}
\left( I-\frac{\zeta \otimes \zeta}{|\zeta|^2}\right)A
+\frac{\zeta \otimes \zeta}{|\zeta|^2}\right).
\end{equation}

As in the case of the mean curvature, it is not difficult to see
that $F$ is elliptic and translation invariant, and that $F$ is
continuous off $\{\zeta=0\}$ (this time the singularities at
$\zeta=0$ are of higher order, as $F(\zeta, A)$ generally tends to
$\pm\infty$ as $\zeta$ goes to $0$).

\bigskip

\begin{ex}\label{mean curvature in arbitrary codimension} {\em \bf Motion by mean curvature in arbitrary codimension.}
\end{ex}

If $\Gamma_{0}$ is a $k$-codimensional surface of an
$n$-dimensional Riemannian manifold $M$, we choose a continuous
function $v_0$ with $\Gamma_0=v_{0}^{-1}(0)$, and consider
    $$
    u_t +F(Du, D^2 u)=0, \, \, \, u(0, x)=v_{0}(x),
    $$
where $$F(\zeta, A)=\sum_{i=1}^{d-k}\lambda_{i}(Q)$$ and $$
\lambda_{1}(Q)\leq \lambda_{2}(Q)\leq ... \leq\lambda_{d-1}(Q)
$$ are the eigenvalues of $Q:=S_{\zeta}A
S_{\zeta}$, with
    $$
    S_\zeta :=\left(
I-\frac{\zeta\otimes \zeta}{|\zeta|^2}\right),
$$
corresponding to eigenvectors orthogonal to $\zeta$ (note that
$\zeta$ is an eigenvector corresponding to the eigenvalue $0$ of
$Q$).

The same proof as in \cite{AS} shows that $F$ is elliptic, the key
observation is that
    $$
    \lambda_{i}(Q)=\max\{\min_{\eta\in E}\frac{\langle Q\eta, \eta\rangle}{|\eta|^{2}} \, : \,
    E\subset TM_x, \textrm{ codim}(E)\leq i-1\}.
    $$
On the other hand, it is easy to see, as in Example \ref{example
mean curvature} above, that $F$ is translation invariant.

\bigskip

Our aim is to establish comparison, existence and uniqueness of
viscosity solutions to the general curvature evolution equation
\ref{curvatureevolutioneq}, and then to prove that the resulting
generalized motion is consistent with the corresponding classical
motion (whenever the latter exists). However, because this equation
is, in general, highly singular, one has to define very carefully
what a viscosity solution to \ref{curvatureevolutioneq} is at points
where $Du=0$. Here we will adapt Ishii-Souganidis' definition
\cite{IS} (see also \cite{Giga}) from the Euclidean to the
Riemannian setting. This requires a slight change in the definition
of the set of test functions $\varphi$.

\begin{defn}\label{def of the F class}
{\em Let $F:J_{0}^{2}(M) \rightarrow \mathbb R$ be continuous,
(degenerate) elliptic, translation invariant and geometric. Denote
by $\mathcal{F}=\mathcal{F}(F)$ the set of functions $f\in
C^{2}([0, \infty))$ such that $f(0)=f'(0)=f''(0)=0$ and $f''(s)>0$
for $s>0$ which satisfy \begin{equation}\label{limit property of
f, F} \lim_{|\zeta|\to 0}\frac{f'(|\zeta|)}{|\zeta|}F(\zeta, 2 I)=
\lim_{|\zeta|\to 0}\frac{f'(|\zeta|)}{|\zeta|}F(\zeta, -2I)=0.
\end{equation}
It is clear that $\mathcal{F}$ is a cone (that is,
$f+g\in\mathcal{F}$ and $\lambda f\in\mathcal{F}$ whenever $f, g\in
\mathcal{F}, \lambda\in [0, \infty$)).}
\end{defn}

\begin{prop}
If $M$ is compact and $F:J_{0}^{2}(M) \rightarrow \mathbb R$ is
continuous, elliptic, translation invariant, and geometric, then
$\mathcal{F}(F)\neq\emptyset$.
\end{prop}
\begin{proof}
One can adapt the proof given in \cite[p. 229]{IS} for the case
$M=\mathbb{R}^{n}$. The only difference (apart from the replacement
of $I$ with $2I$) is that $|\zeta|=|\zeta|_{x}$ depends on the point
$x$ such that $\zeta\in TM_{x}$, and one has to be cautious about
this dependence (as a matter of fact, that is why we require
compactness of $M$). Let us give the essential details for the
reader's convenience.

Since $F$ is continuous on $J_{0}^{2}(M)$ and the sets
$\{(\zeta_{x}, \pm 2I) : |\zeta_{x}|_{x}=1, x\in M\}$ are compact
in $J_{0}^{2}(M)$, there exists a continuous function
$c:(0,\infty)\to (0, \infty)$ such that
$$
-c(|\zeta|)\leq F(\zeta, 2I)\leq F(\zeta, -2I)\leq c(|\zeta|)
$$
for all $\zeta\in TM^*\setminus \{0_x: \, x \in M\}$. Without loss
of generality one can then assume that $c$ is $C^1$ on $(0, \infty)$
and satisfies $(1/c)'>0$ in $(0,1]$, $\lim_{r\to 0^+} c(r)=\infty$,
and $\lim_{r\to 0^+}(1/c)'(r)=0$. Then it is not difficult to show
that an appropriate extension to $[0, \infty)$ of the function $f$
defined on $[0,1]$ by
$$f(r)=\left\{
                                                \begin{array}{ll}
                                                  \int_{0}^{r}\frac{s^{2}}{c(s)} ds, & \hbox{ if $0<r\leq 1$;} \\
                                                  0, & \hbox{ if $r=0$,}
                                                \end{array}
                                              \right.
$$ belongs to $\mathcal{F}(F)$.
\end{proof}

\medskip

For many interesting choices of the function $F$ it is easy to
show that $\mathcal{F}(F)\neq\emptyset$ without requiring $M$ to
be compact:

\begin{ex}
{\em If $F$ is given by \eqref{meancurvature2} (corresponding to the
mean curvature evolution equation), then we may take
$f\in\mathcal{F}(F)$ of the form
$$f(t)=t^4.$$ On the other hand, when $F$ is associated to the Gaussian curvature
evolution equation (that is, $F$ is given by \eqref{F for Gaussian
curvature}) then
$$f(t)=t^{2n}$$
belongs to $\mathcal{F}(F)$ (here $n$ is the dimension of $M$).}
\end{ex}

\begin{defn}\label{def of the A(F) class}
{\em We define the set $\mathcal{A}(F)$ of {\em admissible test
functions} for the equation \eqref{curvatureevolutioneq} as the
set of all functions $\varphi\in C^{2}((0, T)\times M)$ such that,
for every $z_{0}=(t_{0}, x_{0})\in (0,T)\times M$ with
$D\varphi(z_{0})=0$ there exist some $\delta>0$,
$f\in\mathcal{F}$, $w\in C([0, \infty))$ satisfying $\lim_{r\to
0^{+}}w(r)/r=0$ and
    $$
    \left| \varphi(z)-\varphi(z_{0})-\varphi_{t}(z_0)(t-t_{0}) \right|
    \leq f(d(x,x_{0}))+w(|t-t_{0}|)
    $$
for all $z=(t,x)\in B(z_{0}, \delta)$. }
\end{defn}
Notice that in particular, for all $\varphi\in\mathcal{A}(F)$ we
have that $$D\varphi(z)=0\implies D^{2}\varphi(z)=0.$$

\begin{prop}\label{AF is dense}
If $M$ is a compact Riemannian manifold then the class
$\mathcal{A}(F)$ of admissible test functions is dense in the
space $\mathcal{C}(M)$ of continuous functions on $M$.
\end{prop}
\begin{proof}
It is not difficult to check that the class $\mathcal{A}(F)$
satisfies the hypotheses of the Stone-Weierstrass theorem.
\end{proof}

\begin{defn}
{\em We will say that an upper semicontinuous function
$u:[0,T)\times M\to\mathbb{R}$ is a viscosity subsolution of
\eqref{curvatureevolutioneq} provided that, for every
$\varphi\in\mathcal{A}(F)$ and every maximum point $z=(t,x)$ of
$u-\varphi$, we have
    $$
  \begin{cases}
    \varphi_{t}+F(D\varphi(z), D^{2}\varphi(z))\leq 0 & \text{ if } D\varphi(z)\neq 0, \\
    \varphi_{t}(z)\leq 0 & \text{otherwise}.
  \end{cases}
    $$
Similarly, we will say that a lower semicontinuous function
$u:[0,T)\times M\to\mathbb{R}$ is a viscosity supersolution of
\eqref{curvatureevolutioneq} if, for every
$\varphi\in\mathcal{A}(F)$ and every minimum point $z=(t,x)$ of
$u-\varphi$, we have
    $$
  \begin{cases}
    \varphi_{t}+F(D\varphi(z), D^{2}\varphi(z))\geq 0 & \text{ if } D\varphi(z)\neq 0, \\
    \varphi_{t}(z)\geq 0 & \text{otherwise}.
  \end{cases}
    $$
A viscosity solution of \eqref{curvatureevolutioneq} is a continuous
function $u:[0,T)\times M\to\mathbb{R}$ which is both a viscosity
subsolution and a viscosity supersolution of
\eqref{curvatureevolutioneq}. }
\end{defn}

In \cite{Giga} this kind of solution is called an
$\mathcal{F}$-solution, but here we will simply call it a solution.
It is clear that one can always assume that the minimum or maximum
in these definitions are strict.

Notice that the set of test functions $\varphi$ we are using is
smaller than the standard one in the general theory of viscosity
solutions, and that we here require that $\varphi$ is $C^2$ with
respect to the variables $t$ and $x$ (while in the usual definition
of the parabolic semijets one demands $C^1$ differentiability with
respect to $t$ and $C^2$ differentiability with respect to $x$).

It is easy to check that this definition is consistent with $u$
being a classical solution. Indeed, if $u$ is a classical solution
then we have $Du(z)\neq 0$ and $u_{t}(z)+F(Du(z), D^2u(z))=0$ for
all $z$. Then, if $\varphi\in\mathcal{A}(F)$ is such that
$u-\varphi$ attains a minimum at $z$, we have
$\varphi_{t}(z)=u_{t}(z)$, $D\varphi(z)=Du(z)\neq 0$, and
$D^{2}u(z)\geq D^{2}\varphi(z)$. Since $F$ is elliptic we get
    $$
    \varphi_{t}(z)+F(D\varphi(z), D^{2}\varphi(z))\geq
    u_{t}(z)+F(Du(z), D^{2}u(z))=0,
    $$
that is, $u$ is a supersolution at $z$. A similar argument shows
that $u$ is a subsolution.

It can be proved, as in the Euclidean case \cite{Giga}, that if the
lower and upper semicontinuous envelopes of $F$ (denoted by
$\underline{F}$ and $\overline{F}$ respectively) are finite and
$\underline{F}(0,0)=\overline{F}(0,0)=0$, then every standard
viscosity solution is an $\mathcal{F}$-solution, and conversely.
This is the case of the $F$ associated to the mean curvature
evolution.

\bigskip

\section{Some technical tools}

In this section we collect some rather technical results that will
be needed in the proof of the main comparison theorem.

First, we will need to use the following variant of the maximum
principle for semicontinuous functions already used in \cite{AFS},
which we restate here for the reader's convenience.
\begin{thm}\label{key to comparison}
Let $M_1 , ..., M_k$ be Riemannian manifolds, and $\Omega_i\subset
M_i$ open subsets. Define $\Omega=\Omega_1\times \ldots\times
\Omega_n\subset M_1\times\ldots\times M_k=M$. Let $u_i$ be upper
semicontinuous functions on $\Omega_i$, $i=1,...,k$; let $\varphi$
be a $C^2$ smooth function on $\Omega$ and set
    $$\omega(x)=u_1(x_1)+\ldots+u_n(x_k)$$
for $x=(x_1,...,x_k)\in\Omega$. Assume that
$(\hat{x}_1,\ldots,\hat{x}_k)$ is a local maximum of
$\omega-\varphi$. Then, for each $\varepsilon>0$ there exist
bilinear forms $B_i \in
\mathcal{L}^{2}_{s}((TM_{i})_{\hat{x}_{i}})$,  $i=1, ..., k$, such
that
$$\left(D_{x_{i}}\varphi(\hat{x}),B_i \right)\in \overline{J}^{\, 2, +}u_i(\hat{x}_i)$$
for $i=1,...,k$, and the block diagonal matrix with entries $B_i$
satisfies
    $$-\left({1\over\varepsilon}+\|A\|\right)I\leq
    \left(
\begin{array}{ccc}
  B_1   & \ldots &   0    \\
 \vdots & \ddots & \vdots \\
   0    & \ldots &  B_k   \\
\end{array}\right)\leq A+\varepsilon A^2,$$
where $A=D^2\varphi(\hat{x})\in \mathcal{L}_s^2(TM_{\hat{x}})$.
\end{thm}
\noindent We recall that $$
    J^{2,+}f(x)=\{(D\varphi(x), D^{2}\varphi(x)) \, : \,
    \varphi\in C^{2}(M, \mathbb{R}), \,  f-\varphi \textrm{
    attains a local maximum at } x\},
    $$
and
$$
\overline{J}^{2,+}f(x)=\{(\zeta, A)\in
TM^{*}_{x}\times\mathcal{L}_{s}(TM_x)
     \, : \, \exists (x_{k},\zeta_{k}, A_{k})\in M\times
     TM^{*}_{x_{k}}\times\mathcal{L}_{s}(TM_{x_{k}})$$ $$  s.t. \,\, (\zeta_{k}, A_{k})\in
     J^{2,+}f(x_{k}), \,\, (x_{k}, f(x_{k}), \zeta_{k}, A_{k})\to (x,f(x),\zeta, A)\},
     $$
see \cite{AFS}.

Another important ingredient of the proof of our main comparison
result is the following Proposition, established in
\cite[Proposition 3.3]{AFS}.

\begin{prop}\label{bound for A with no restriction on curvature}
Consider the function $\Psi(x,y) = d(x,y)^{2}$, defined on
$M\times M$. Assume that the sectional curvature $K$ of $M$ is
bounded below, say $K\geq -K_{0}$. Then
    $$
    D^{2}_{x,y}\Psi(x,y)(v, L_{xy}v)^{2}\leq 2 K_{0}d(x,y)^{2} \|v\|^{2}
    $$
for all $v\in TM_{x}$ and $x,y\in M$ with $d(x,y)<\min\{ i_{M}(x),
i_{M}(y)\}$.

In particular, if $-K_{0}\geq 0$ (that is $M$ has nonnegative
sectional curvature) one has that the restriction of
$D^{2}_{x,y}\Psi(x,y)$ to the subspace $\mathcal{D}=\{(v, L_{xy}v)
: v\in TM_{x}\}$ of $TM_{x}\times TM_{y}$ is negative
semidefinite.
\end{prop}

We will also need the following auxiliary result.

\begin{lem}\label{lemma-m-alphas} Let $\phi \in USC(M)$,
$\psi \in LSC(M)$, $f\in\mathcal{F}(F)$, and
    $$m_{\alpha}:=\sup_{ M\times M}
    \left(\phi(x)-\psi(y)-\alpha\,f\left(d(x,y)^2\right)\right)$$
for $\alpha>0.$ Suppose $m_{\alpha}<\infty$ for large $\alpha$ and
let $(x_{\alpha}, y_{\alpha})$ be such that
    $$\lim_{\alpha\rightarrow \infty}\left(m_{\alpha}-(\phi(x_{\alpha})-
    \psi(y_{\alpha})-\alpha f(d(x_{\alpha},y_{\alpha})^2))\right)=0.$$ Then we have:
\begin{enumerate}
\item $\lim_{\alpha\rightarrow \infty}\alpha
                  f(d(x_{\alpha},y_{\alpha})^2)=0$, and \smallskip
\item if $\widehat{x}\in M$ is a
limit point of $x_{\alpha}$ as $\alpha\rightarrow\infty$ then
\[
\lim_{\alpha\rightarrow\infty}m_{\alpha}=\phi(\widehat{x})-\psi(\widehat{x})
=\sup_{x\in M}(\phi(x)-\psi(x)).
\]
\end{enumerate}
\end{lem}
\begin{proof}
A more general form of this result is proved in \cite[Theorem
3.7]{CIL} in the case when $M$ is an Euclidean space, and the same
proof clearly works in a general metric space.
\end{proof}

Let us now define $\mathcal P^{2,+}, \, \mathcal P^{2,-}$,
$\mathcal{ \overline{P}}^{2,+}, $ and  $\mathcal{
\overline{P}}^{2,-}$,  the ``parabolic" variants of the semijets
$J^{2,+}, \,  J^{2,-}$, $ \overline{J}^{2,+}, \,
\overline{J}^{2,-}$ introduced in \cite{AFS} for functions defined
on a Riemannian manifold.

\begin{defn}
{\em Let $f:(0,T)\times M\to (-\infty, +\infty]$ be a lower
semicontinuous (LSC) function. We define the parabolic second
order subjet of $f$ at a point $(t_0,x_0)\in (0,T)\times M$ by
\begin{align*}
    \mathcal{P}^{2,-}f(t_0,x_0):=&\{(D_t \varphi(t_0,x_0), D_x\varphi(t_0,x_0)),
     D^{2}_x\varphi(t_0,x_0)) \, : \,
  \varphi \text{ is once continuously} \\  & \text{differentiable in }  t\in (0,T),
  \text{ twice continuously differentiable in } x\in M
 \\ & \text{and }  f-\varphi  \textrm{
    attains a local minimum at } (t_0,x_0)\}.
    \end{align*}
   Similarly, for an upper semicontinuous (USC) function $f:(0,T)\times M\to[-\infty,
 +\infty)$, we define the parabolic second order superjet of $f$ at $(t_0,x_0)$ by
 \begin{align*}
 \mathcal{P}^{2,+}f(t_0,x_0):= & \{(D_t \varphi(t_0,x_0), D_x\varphi(t_0,x_0)),
  D^{2}_x\varphi(t_0,x_0)) \, : \,
  \varphi \text{ is once continuously} \\  & \text{differentiable in }  t\in (0,T),
   \text{ twice continuously differentiable in } x\in M
   \\ & \text{and }
f-\varphi  \textrm{  attains a local maximum at } (t_0,x_0)\}.
  \end{align*}
}
\end{defn}

Observe that $\mathcal{P}^{2,-}f(t,x)$ and
$\mathcal{P}^{2,+}f(t,x)$ are subsets of $\mathbb R\times
TM^{*}_{x}\times\mathcal{L}^{2}_{s}(TM_{x})$. Notice that we can
assume that the  auxiliary functions $\varphi$ are defined on a
neighborhood of $(t_0,x_0)$. We may as well assume (just by adding
a function of the form $\pm \varepsilon d(x, x_{0})^4$) that the
minima or maxima in these definitions are strict. It is also
easily seeing that the min or max can always be supposed to be
global.
\begin{defn}
\emph{Let $f:(0,T)\times M\longrightarrow(-\infty,+\infty]$ be a
LSC function and $(t,x)\in(0,T)\times M$. We define
$\mathcal{\overline{P}}^{2,-}f(t,x)$ as the set of the
$(a,\zeta,A)\in\mathbb{R}\times TM_{x}^{\ast}\times
\mathcal{L}_{s}^{2}(TM_{x})$ such that there exist a sequence $(x_{k}%
,a_{k},\zeta_{k},A_{k})$ in $M\times\mathbb{R}\times TM_{x_{k}}^{\ast}%
\times\mathcal{L}_{s}^{2}(TM_{x_{k}})$ satisfying:}%
\[%
\begin{array}
[c]{ll}%
\text{\emph{i)}} & (a_{k},\zeta_{k},A_{k})\in\mathcal{P}^{2,-}f(t_{k}%
,x_{k}),\medskip\\
\text{\emph{ii)}} & \lim_{k}(t_{k},x_{k},f(t_{k},x_{k}),a_{k},\zeta_{k}%
,A_{k})=(t,x,f(t,x),a,\zeta,A).
\end{array}
\]
\end{defn}
\noindent The corresponding definition of $
\mathcal{\overline{P}}^{2,+}f(t,x)$ when $f$ is an upper
semicontinuous function is then clear.

\medskip

The next two lemmas are needed to establish the parabolic version
of the maximum principle we state below.

\begin{lem}[\cite{AFS}] \label{exp}
Let $U\subset M$ be an open subset, $(t,z)\in  (0,T)\times U$ and
a function \, $\varphi:(0,T)\times U \to \mathbb{R}$. Assume that
$\varphi$ is once continuously differentiable in $(0,T)$ and twice
continuously differentiable in $U$. Define
${\psi}(s,v)=\varphi(s,\exp_z v)$ on a neighborhood of $0\in
TM_{z}$. Let $\widetilde{V}$ be a vector field defined on a
neighbourhood of $0$ in $TM_z$, and consider the vector field
defined by $V(y)=D \exp_z(w_y)(\widetilde{V}(w_y))$ on a
neighbourhood of $z$, where $w_y:=\exp^{-1}_z(y)$, and let
 $$
    \sigma_{y}(r)=\exp_{z}(w_{y}+r\widetilde{V}(w_{y})).
    $$
Then we have that
    $$
    D^{2}_v\psi(\widetilde{V},\widetilde{V})(t,w_{y})=
    D^{2}_x\varphi(V,V)(t,y)+\langle \nabla_x\varphi(t,y),
    \sigma_{y}''(0)\rangle.
    $$

\noindent Observe that $\sigma_{z}''(0)=0$ so, when $y=z$, we obtain
    $$
    D^{2}_v\psi(t,0)=D^{2}_x\varphi(t,z).
    $$
\end{lem}

\begin{proof}
Analogous to \cite[Lemma 2.7]{AFS}.
\end{proof}

\medskip

\begin{lem}\label{parabolic} Let $U\subset M$ be an open
subset, $(t,z)\in  (0,T)\times U$ and  $u:(0,T)\times M\to
[-\infty,\infty)$ be an  upper semicontinuous function and consider
a neighbourhood $V$ of \ $0\in TM_z$ and $\widetilde{u}:(0,T)\times
V\to [-\infty,\infty)$ defined as $\widetilde{u}(s,v)=u(s,\exp_z
v)$. Then, if $(b,\zeta,A)\in \mathbb R \times TM^*_z\times
\mathcal{L}_s^2(TM_z)$,
 $$
    (b,\zeta, A)\in \overline{\mathcal{P}}^{2,+}u(t,z) \iff (b,\zeta, A)\in
    \overline{\mathcal{P}}^{2,+}\widetilde{u}(t,0).
    $$
\end{lem}

\begin{proof}
Use the above Lemma as in the proof of \cite[Proposition 2.8]{AFS}.
\end{proof}

\medskip

As in \cite{Ilmanen} in the case of the mean curvature evolution
equation and \cite{AFS} in the case of general (nonsingular)
stationary equations, the following result is one of the keys to
the proof of the comparison result for general (nonsingular)
evolution equations which we give in the Appendix.

\begin{thm}\label{keyparabolic}
Let $M_1, ...,M_k$ be Riemannian manifolds, and $\Omega_i\subset
M_i$ open subsets. Define $\Omega=(0,T)\times\Omega_1\times
\cdots\times \Omega_k$. Let $u_i$ be upper semicontinuous functions
on $(0,T)\times \Omega_i$, $i=1,...,k$; let $\varphi$ be a function
defined on $\Omega$ such that it is once continuously differentiable
in $t\in (0,T)$ and twice continuously differentiable in
$x:=(x_1,...,x_k)\in \Omega_1\times \cdots \times \Omega_k$ and set
    $$\omega(t,x_1,...,x_k)=u_1(t,x_1)+\cdots+u_k(t,x_k)$$
for $(t,x_1,...,x_k)\in\Omega$. Assume that
$(\widehat{t},\widehat{x}_1,\ldots,\widehat{x}_k)$ is a maximum of
$\omega-\varphi$ in $\Omega$. Assume, moreover, that there is an
$\tau>0$ such that for every $M>0$ there is $C>0$ such that for
$i=1,...,k$,
\begin{equation}  \label{boundconditions}   \begin{cases}
a_i\le C \text{ \ whenever \ } (a_i,\zeta_i,A_i)\in \mathcal{
\overline{P}}^{2,+}_{M_i}u_i(t,x_i) &\\
d(x_i,\widehat{x}_i)+|t-\widehat{t}|\le \tau \text{ \ and \  }
|u_i(t,x_i)|+|\zeta_i|+||A_i||\le M.&
\end{cases}\end{equation}

Then, for each $\varepsilon>0$ there exist real  numbers $b_i$ and
bilinear forms $B_i \in
\mathcal{L}^{2}_{s}((TM_{i})_{\widehat{x}_{i}})$,  $i=1, ..., k$,
such that
$$
\left(b_i,
D_{x_{i}}\varphi(\widehat{t},\widehat{x}_1,...,\widehat{x}_k), B_i
\right)\in \mathcal{\overline{P}}^{\, 2, +}_{M_i} u_i(\widehat{t},
\widehat{x}_i)$$ for $i=1,...,k$, and  the block diagonal matrix
with entries $B_i$ satisfies
    $$-\left({1\over\varepsilon}+\|A\|\right)I\leq
    \left(
\begin{array}{ccc}
  B_1   & \ldots &   0    \\
 \vdots & \ddots & \vdots \\
   0    & \ldots &  B_k   \\
\end{array}\right)\leq A+\varepsilon A^2,$$
where $A=D_x^2\varphi(\widehat{t},\widehat{x}_1,...,\widehat{x}_k)$
 \, and \, $b_1+\cdots+b_k=\frac{\partial \varphi}{\partial t}(\widehat{t},\widehat{x}_1,...,\widehat{x}_k)$.
\end{thm}
\begin{proof} The result is proved in \cite{CIL} for $M_i=\mathbb R^{n_i}$, $i=1,...,k$. As
 in the stationary case \cite{AFS},
we can reduce the problem to this situation by an adecuate
composition with the exponential
 mappings. Let us give some details for completeness. We may assume
 (by taking  smaller neighborhoods of
 $x_i$, if necessary),  that the sets $\Omega_i$ are
diffeomorphic images of balls by the exponential mappings
$\exp_{\widehat{x}_{i}}:B(0_{\widehat{x}_i},
r_i)\to\Omega_i:=B(\widehat{x}_{i}, r_{i})$. Consider the Riemannian
manifold $M:=M_1\times \cdots \times M_k$ and  $\widehat{x}:=
(\widehat{x}_1,...,\widehat{x}_k)\in \Omega_1\times \cdots \times
\Omega_k$. Recall that if $v:=(v_1,...,v_k)\in B(0_{\widehat{x}_1},
r_1)\times \cdots \times B(0_{\widehat{x}_k}, r_k)$  the exponential
map $\exp_{\widehat{x}}$ is defined as  $\exp_{\widehat{x}}(v) =
    \left(\exp_{\hat{x}_1}(v_{1}), ..., \exp_{\hat{x}_k}(v_{k})\right)$.
     Then
$\exp_{\widehat{x}}$ maps diffeomorphically the open set
$B(0_{\widehat{x}_1}, r_1)\times \cdots \times B(0_{\widehat{x}_k},
r_k)\subset TM_{\widehat{x}}=(TM_{1})_{\widehat{x}_1}\times \cdots
\times (TM_{k})_{\widehat{x}_k}$ onto $\Omega_1\times \cdots \times
\Omega_k$.

We consider the functions, defined on suitable open subsets of
euclidean spaces,
\begin{equation*}
\widetilde{\omega}(t,v):=\omega(t,\exp_{\widehat{x}}(v)), \quad
\psi(t,v):=\varphi(t,\exp_{\widehat{x}}(v)), \quad
\widetilde{u}_{i}(t,v_i):=u_{i}(t,\exp_{\widehat{x}_i}(v_i)).
\end{equation*}
We have that $$\widetilde{\omega}(t,v_1,..., v_k)=\widetilde{u}_1
(t,v_1)+\cdots+\widetilde{u}_k (t,v_k),$$ and
$(\widehat{t},0_{\widehat{x}})=(\widehat{t},0_{\widehat{x}_1},...,0_{\widehat{x}_{k}})$
is the maximum of $\widetilde{\omega}-\psi$. Therefore, we apply
\cite[Theorem 8.3]{CIL} to ensure, for every $\varepsilon>0$, the
existence of  real  numbers $b_i$ and bilinear forms $B_i \in
\mathcal{L}^{2}_{s}(\mathbb R^{n_i})$,  $i=1, ..., k$, such that
$$
\left(b_i,D_{v_{i}}\psi(\widehat{t},0_{\widehat{x}}), B_i
\right)\in \mathcal{\overline{P}}^{\, 2, +}
 \widetilde{u}_i(\widehat{t}, 0_{\widehat{x}_i})$$
for $i=1,...,k$, and  the block diagonal matrix with entries $B_i$
satisfies
    $$-\left({1\over\varepsilon}+\|A\|\right)I\leq
    \left(
\begin{array}{ccc}
  B_1   & \ldots &   0    \\
 \vdots & \ddots & \vdots \\
   0    & \ldots &  B_k   \\
\end{array}\right)\leq A+\varepsilon A^2,$$
where $A=D_v^2\psi(\widehat{t},0_{\widehat{x}})$
 \, and \, $b_1+\cdots+b_k=
 \frac{\partial \psi}{\partial t}(\widehat{t},0_{\widehat{x}})$.
Clearly
\begin{equation*}
\frac{\partial \psi}{\partial t}(\widehat{t},0_{\widehat{x}})=
\frac{\partial \varphi}{\partial t}(\widehat{t},\widehat{x}),
\quad \quad D_{v_i}\psi(\widehat{t},0_{\widehat{x}})=
D_{x_i}\varphi(\widehat{t},\widehat{x}),\end{equation*}
  and Lemma \ref{exp}
provides the equality $D_v^2\psi(\widehat{t},0_{\widehat{x}})=
D_x^2\varphi(\widehat{t},\widehat{x})$. To conclude this proof it
remains to apply Lemma \ref{parabolic} to obtain the equivalence
\begin{equation*} \left(b_i,D_{v_i}\psi(\widehat{t},0_{\widehat{x}}),B_i \right)\in
\mathcal{\overline{P}}^{\, 2,
+}\widetilde{u}_i(\widehat{t},0_{\widehat{x}_i}) \, \, \iff \, \,
\left(b_i, D_{x_i}\varphi(\widehat{t},\widehat{x}),B_i \right)\in
\mathcal{\overline{P}}^{2, +} u_i(\widehat{t},\widehat{x}_i).
\end{equation*}

\end{proof}

\bigskip

\section{Comparison}

Let us state and prove our main comparison result  for viscosity
solutions of \eqref{curvatureevolutioneq}.

\begin{thm} \label{compact} Let $M$ be a compact Riemannian manifold
of nonnegative sectional curvature, and let $F:J_{0}^{2}(M)
\rightarrow \mathbb R$ be continuous, elliptic, translation
invariant and geometric. Let $u\in USC([0,T)\times M)$ be a
subsolution and $v\in LSC([0,T)\times M)$ be a supersolution of
\eqref{curvatureevolutioneq} on $M$.
   Then $u \le v$ on $[0,T)\times M$ whenever
$u \le v$ on $\{0\}\times M$.
 \end{thm}

 \begin{proof} Since $M$ is compact we know that $M$ has injectivity radius $i_M>0$.

Let us start noting that we may assume $u$ and $-v$ bounded above on
$[0,T)\times M$. Otherwise, for every $0<S<T$, consider $u$ and $-v$
defined on the compact set $[0,S]\times M$, where they are also
u.s.c. and thus bounded above. Then, we apply the arguments of the
proof  to $u$ and $-v$ in $[0,S)\times M$.

Next, let us observe that for $\varepsilon >0$, the function
$\widetilde{u}=u-\frac{\varepsilon}{T-t}$
 is also a subsolution of \, $u_t+F(Du,D^2u)=0$ on  $[0,T)\times M$.
 Moreover,
\begin{eqnarray}\label{conditions on u}
& &\widetilde{u}_t+F(D\widetilde{ u}, D^2 \widetilde{ u})\le -
\frac{\varepsilon}{T^2} \quad \text{ for } D\widetilde{u}\neq 0,\\
& &\widetilde{u}_t\leq -\frac{\varepsilon}{T^2} \quad \text{ for }
D\widetilde{u}=0, \text{ and } \\
& & \lim_{t\to T^-}\widetilde{ u}(t,x)=-\infty \text{ uniformly on
} M.
 \end{eqnarray}
 Since the  assertion $\widetilde{u}\le v$ for every $\varepsilon>0$ implies $u\le v$,
  it will suffice to prove the comparison result under the assumptions given in (4.1-4.3).

Assume to the contrary that $\displaystyle{\sup_{[0,T)\times
M}(u-v)>0}$. Take $f\in\mathcal{F}$.
  Since  $M$ is compact, $u$ and  $ -v$ are u.s.c. and (4.3) holds, we can consider for every $\alpha \in \mathbb{N}$,
 $$m_\alpha:=\sup_{\substack{0\le s, t<T \\ x,y \in M}}\{u(s,x)-v(t,y)-\alpha f\left(d(x,y)^2\right)-\alpha(t-s)^2\}, $$
 which is attained at some $(s_{\alpha}, t_{\alpha}, x_\alpha, y_\alpha)\in [0,T)\times [0,T)\times M\times M$.
Clearly,
    $$m_\alpha \ge \displaystyle{\sup_{[0,T)\times M}(u-v)>0}.$$
If $t_\alpha=0$ for infinitely many $\alpha$'s, which we may assume
are all $\alpha$, then we have
 $$
 0<\sup_{[0,T)\times M}(u-v)\le m_\alpha = \sup_{s,x,y}\left(u(s,x)-v(0,y)-\alpha f\left(d(x,y)^2\right)-
 \alpha s^2\right)
 $$We deduce from  Lemma \ref{lemma-m-alphas}  that
 $\lim_{\alpha\rightarrow \infty}\alpha f\left(d(x_\alpha,y_\alpha)^2\right)=0$ and
 $\lim_{\alpha\rightarrow \infty}\alpha (t_{\alpha}-s_{\alpha})^2=0$. By compactness, we can assume that a
  subsequence of $(t_{\alpha},s_{\alpha}, x_\alpha,y_\alpha)$, which we still denote
  $(t_{\alpha}, s_{\alpha}, x_\alpha,y_\alpha)$, converges to a point $(s_{0}, t_{0}, x_0,y_0)$. By Lemma
  \ref{lemma-m-alphas} we have that $x_0=y_0$ and $s_{0}=t_{0}=0$,
  and
  $\lim_{\alpha\rightarrow\infty}m_{\alpha}=u(0,x_0)-v(0,x_0)
=\sup_{x\in M}(u(0,x)-v(0,x))\le 0$, which is a contradiction.

\medskip

A completely analogous argument leads us to a contradiction if
$s_{\alpha}=0$ for infinitely many $\alpha$'s.

\medskip

Thus we may assume that there exist $\alpha_0>0$ such that
$s_{\alpha}>0$ and $t_{\alpha}>0$ for $\alpha>\alpha_0$. By
compactness and Lemma \ref{lemma-m-alphas} we may also assume that
$x_{\alpha}$ and $y_{\alpha}$ converge to the same point
$x_{0}=y_{0}$, and in particular that $x_\alpha,y_\alpha \in
B(x_0,r/2)$ for all $\alpha>\alpha_0$, where $r>0$ is small enough
such that $0<r<i_M$ and conditions ({\bf B}, {\bf C}) of Section 2
hold whenever $d(x,y)<r$. Therefore the function $d(x,y)^2$ and
hence the functions
\[
\Phi_\alpha(x,y):=\alpha f(d(x,y)^2),\qquad
\varphi_\alpha(s,t,x,y):=\Phi_\alpha(x,y)+\alpha(t-s)^2
\]
are $C^2$ smooth on $(0,T)\times (0,T)\times B(x_0,r/2)\times
B(x_0,r/2) $.

Recall that $\mathcal{\overline{P}}^{2,-}v(t_\alpha,y_{\alpha})=
-\mathcal{\overline{P}}^{2,+}(-v)(t_\alpha,y_{\alpha})$, and if we
consider the function
$$
\Psi(x,y):=d(x,y)^2
$$
we obtain from \cite[Section 3]{AFS} that
  $$  D_x\Psi(x_\alpha,y_\alpha)=-2\exp_{x_\alpha}^{-1}(y_\alpha), \,
  \textrm{
  and } \,\, D_y\Psi(x_\alpha,y_\alpha)=-2\exp_{y_\alpha}^{-1}(x_\alpha).$$

Now we cannot directly apply Theorem \ref{keyparabolic}, because
condition \eqref{boundconditions} is not generally satisfied due
to the singularity of $F$ (one has a serious difficulty when
$\mathcal{\overline{P}}^{\, 2, +}u(s_\alpha,x_{\alpha})$ contains
triplets of the form $(a, 0, A)$: in this case one cannot use the
fact that $u$ is a subsolution to guarantee that $a\leq C$, since
$F(\zeta, A)\to\infty$ as $\zeta\to 0$). Instead we will use
Theorem \ref{key to comparison}, treating the variables $s,t$ as
if they were spatial variables in the stationary case, and then
ignoring the information that this result gives about the second
derivatives with respect to the variables $t,s$, which we do not
need here. Bearing in mind that $(s_{\alpha}, t_\alpha, x_\alpha,
y_\alpha) $ is the maximum of the function $(s,t,x,y) \rightarrow
u(s,x)-v(t,y)-\varphi_\alpha(s,t,x,y)$, and setting
$$
A_{\alpha}:=D^2_{x,y}\varphi_\alpha(s_\alpha, t_\alpha,
x_{\alpha}, y_{\alpha}), \qquad
\varepsilon:=\varepsilon_\alpha=\frac{1}{\,1+||A_\alpha||\,},
$$ we obtain this way two bilinear forms $P_\alpha \in \mathcal{L}^{2}_{s}(TM_{x_{\alpha}})$,
and $Q_\alpha\in \mathcal{L}^{2}_{s}(TM_{y_{\alpha}})$ such that
\begin{eqnarray}\label{subjets and superjets}
& & \left(\frac{\partial}{\partial
s}\varphi_\alpha(s_{\alpha},t_\alpha,
x_{\alpha},y_{\alpha}),D_x\varphi_\alpha(s_{\alpha},t_\alpha,
x_{\alpha},y_{\alpha}), P_\alpha \right) \in
\mathcal{\overline{P}}^{\, 2,
+}u(s_\alpha,x_{\alpha}),\\
& & \left(-\frac{\partial}{\partial
t}\varphi_\alpha(s_{\alpha},t_\alpha, x_{\alpha},y_{\alpha}),
-D_y\varphi_\alpha(s_{\alpha},t_\alpha, x_{\alpha},y_{\alpha}),
Q_\alpha \right) \in \mathcal{\overline{P}}^{\, 2,
-}v(t_\alpha,y_{\alpha}),
\end{eqnarray}
and
\begin{equation} \label{ineqcuadratic}
\quad -\left({1\over\varepsilon_{\alpha}}+\|A_{\alpha}\|\right)I\leq
    \left(
\begin{array}{ccc}
  P_\alpha   &   0    \\
   0  &  -Q_\alpha  \\
\end{array}\right)\leq A_{\alpha}+\varepsilon_{\alpha} A_{\alpha}^2. \end{equation}
These inequalities can be deduced from the corresponding ones in
Theorem \ref{key to comparison} (just bear in mind the special form
of our function $\varphi_{\alpha}$, and apply the inequalities given
by Theorem \ref{key to comparison} to vectors of the form $(0,0,
v,w)$, where the zeros correspond to the variables $s$ and $t$).

In our case we have \begin{eqnarray}\label{expressions for the
subderivatives} & & a_{\alpha}:=\frac{\partial}{\partial
s}\varphi_\alpha(s_{\alpha},t_\alpha,
x_{\alpha},y_{\alpha})=-2\alpha(t_{\alpha}-s_{\alpha}),\\
& & -b_{\alpha}:=-\frac{\partial}{\partial
t}\varphi_\alpha(s_{\alpha},t_\alpha,
x_{\alpha},y_{\alpha})=-2\alpha(t_{\alpha}-s_{\alpha}), \notag\\
& & D_x\varphi_\alpha(s_{\alpha},t_\alpha,
x_{\alpha},y_{\alpha})=-2\alpha f'(d(x_\alpha,y_\alpha)^2)
\exp_{x_{\alpha}}^{-1}(y_{\alpha}), \\
& & -D_y\varphi_\alpha(s_{\alpha},t_\alpha,
x_{\alpha},y_{\alpha})=2\alpha f'(d(x_\alpha,y_\alpha)^2)
\exp_{y_{\alpha}}^{-1}(x_{\alpha}), \notag
\end{eqnarray}
and in particular we see that
\begin{equation}\label{a+b=0}
a_{\alpha}+b_{\alpha}=0.
\end{equation}

Let us now distinguish two cases.

\noindent {\bf Case 1.} Assume that  $x_\alpha\not=y_\alpha$. Let
us consider the non-zero vectors
$$\zeta_\alpha:=-2\alpha f'(d(x_\alpha,y_\alpha)^2)
\exp_{x_{\alpha}}^{-1}(y_{\alpha}),$$ and notice that $$L_{x_\alpha
\, y_\alpha} \zeta_\alpha=2\alpha f'(d(x_\alpha,y_\alpha)^2)
\exp_{y_{\alpha}}^{-1}(x_{\alpha}).$$

Since $u$ is a strict subsolution and $v$ is a supersolution of
$u_t+F(Du,D^2u)=0$,  we have that
\begin{equation*}
    a_\alpha+ F(\zeta_{\alpha},
P_\alpha)
    \leq\frac{-\varepsilon}{T^2}<0\leq
    -b_\alpha+ F(L_{x_\alpha
\, y_\alpha} \zeta_\alpha, Q_\alpha)
\end{equation*}
(notice that here we used continuity of $F$ off $\{\zeta=0\}$, and
the important observation that if
$(\zeta,A)\in\mathcal{\overline{P}}^{2,+}u(z)$ and $\zeta\neq 0$
then $(\zeta,A)$ is a limit of a sequence $(\zeta_{k}, A_{k})$ with
$(\zeta_{k}, A_{k})=(D\varphi_{k}(z_{k}), D^{2}\varphi_{k}(z_{k}))$
for some $\varphi_{k}\in\mathcal{A}(F)$ and $z_{k}\to z$).

Thus, there is $c:=\frac{\varepsilon}{T^2}$ such that
\begin{equation} \label{pre-Funder-Fupper}
    0<c\le F(L_{x_\alpha
\, y_\alpha} \zeta_\alpha ,Q_\alpha)-
    F(\zeta_\alpha, P_\alpha).
\end{equation}
On the other hand, since $F$ is translation invariant, we deduce
\begin{equation}  \label{Funder-Fupper}
    0<c \le F(L_{x_\alpha \, y_\alpha} \zeta_\alpha,Q_\alpha)-
F(\zeta_\alpha, P_\alpha) = F(\zeta_\alpha, L_{y_\alpha \,
x_\alpha}( Q_\alpha))-
    F(\zeta_\alpha, P_\alpha)
\end{equation}

\medskip

Recall that $A_\alpha=D^2\varphi_\alpha (s_\alpha, t_\alpha,
x_\alpha, y_\alpha)$ and $\Phi_\alpha=\alpha f\circ\Psi$. Then
\begin{align}
D_{x,y}\Phi_\alpha(x, y) &=\alpha f'(\Psi(x, y))\left(D_x\Psi(x,y),D_y\Psi(x,y)\right)= \\
&= -2\alpha f' (\Psi(x, y)) \left( \exp_x^{-1}y ,
\exp_y^{-1}x\right) \notag.
\end{align}
Now $D^2_{x,y}\varphi_\alpha$, the Hessian of $\varphi_\alpha$,
satisfies for every
 vector fields $X,Y$ on $M\times M$
\begin{align}
\label{D^2}D^2_{x,y}\varphi_\alpha(s,t,X,Y)&=\langle \nabla_X(
D\varphi_\alpha), Y \rangle =
\langle \nabla_X (\alpha f'(\Psi) D\Psi), \, Y \rangle \\
&=\alpha\langle  f'( \Psi ) \nabla_X(D\Psi)+ X(f'(\Psi))D\Psi, \,Y  \rangle  \notag \\
&= \alpha f'(\Psi)\langle \nabla_{X}(D\Psi), Y \rangle +\alpha X(f'(\Psi))\langle D\Psi, Y\rangle \notag \\
&= \alpha f'(\Psi) D^2_{x,y}\Psi(X,Y) +\alpha f''(\Psi)X(\Psi) \langle D\Psi, Y\rangle \notag \\
&= \alpha f'(\Psi) D^2_{x,y}\Psi(X,Y) +\alpha f''(\Psi) (D \Psi
\otimes D \Psi)(X,Y). \notag
\end{align}
In particular for every two points $x,y\in M$ such that
$d(x,y)<\min\{i_M(x),i_M(y)\}$ and every  $v\in TM_x$, we consider
$X=Y$ with $X(x,y)=(v,L_{xy}v)\in TM_x\times TM_y$ and we obtain
 \begin{align*}
 X(\Psi)(x,y)=D^2_{x,y}\Psi(x,y)(v,L_{xy}v)=D_x\Psi(x,y)(v)+
 D_y\Psi(x,y)(L_{xy}v)=0.
 \end{align*}
The last equality in the above expression is proved in \cite[Section
3]{AFS}. Therefore, if $M$ has  sectional curvature bounded below by
some constant $-K_0\le 0$, we obtain from equation (5.5) and
Proposition \ref{bound for A with no restriction on curvature} that
 \begin{align}\label{estimation of A}
 A_{\alpha} (s_\alpha, t_\alpha, x_{\alpha},y_{\alpha})(v,L_{x_{\alpha}y_{\alpha}}v)^2 &=
 D^2 {\varphi}_{\alpha}(s_\alpha, t_\alpha, x_{\alpha},y_{\alpha})
 (v,L_{x_{\alpha} y_{\alpha} } v)^2  \notag \\
 &= \alpha f'( \Psi(x_{\alpha},y_{\alpha}))
  D^2\Psi ( x_{\alpha} , y_{\alpha} )
 (v,L_{x_{\alpha} y_{\alpha}} v)^2 \notag \\
 &\le \alpha f'(\Psi(x_{\alpha},y_{\alpha}))2K_0 \Psi(x_{\alpha},y_{\alpha})||v||^2,
 \end{align}
 for every $v\in TM_{x_\alpha}$.
Let us denote by  $\lambda_1 \leq \cdots  \leq \lambda_n$  the
eigenvalues of the restriction of $A_{\alpha}$ to the subspace
$\mathcal{D}=\{(v, L_{x_\alpha y_\alpha}v) : v\in TM_{x_\alpha}\}$
of $TM_{x_\alpha}\times TM_{y_\alpha}$. The above inequality implies
that $\lambda_1,...,\lambda_n \le
 2\alpha K_0  \Psi(x_{\alpha},y_{\alpha}) f'(\Psi(x_{\alpha},y_{\alpha}))$.
 With our choice of
$\varepsilon_{\alpha}$, we have that
    $$
    \lambda_{i}+\varepsilon_{\alpha}\lambda_{i}^{2}\leq
    \lambda_{i}+\frac{1}{\,1+\sup_{1\leq j\leq n}|\lambda_{j}|\,}\,\lambda_{i}^{2}\leq
    \lambda_{i}+|\lambda_{i}|\le 2 \max\{0,\lambda_n \}, \quad i=1,...,n.
    $$
Since $\lambda_{i}+\varepsilon_{\alpha}\,\lambda_{i}^{2}$, $i=1,
..., n$, are the eigenvalues of
$\left(A_{\alpha}+\varepsilon_{\alpha}A_{\alpha}^{2}\right)|_{\mathcal{D}}$,
this means that when $M$ has nonnegative sectional curvature, that
is $K_0=0$, or equivalently $\lambda_n\le 0$, we have
\begin{equation*}
\left(A_{\alpha}+\varepsilon_{\alpha}A_{\alpha}^{2}\right)(v,L_{x_\alpha
y_\alpha}v)^{2}\leq 0.\end{equation*} Therefore, the second
inequality in \eqref{ineqcuadratic} implies $P_\alpha -L_{y_\alpha
\, x_\alpha}(Q_\alpha) \le
(A_{\alpha}+\varepsilon_{\alpha}A_{\alpha}^{2})|_{\mathcal{D}}\le
0$. Thus equation \eqref{Funder-Fupper}, and the fact that $F$ is
elliptic imply that
$$
  0<c \le  F( \zeta_\alpha,L_{y_\alpha \, x_\alpha}(Q_\alpha))-F(\zeta_\alpha, P_\alpha)\le 0,
$$
a contradiction.

\medskip

\noindent {\bf Case 2.} If we are not in Case 1 then we may assume
$x_\alpha=y_\alpha$ for every $\alpha>\alpha_0$. We know that $$
u(s,x)-v(t,y)-\alpha f(d(x,y)^2)-\alpha(t-s)^2\leq
u(s_{\alpha},x_{\alpha})-v(t_{\alpha},y_{\alpha})-\alpha(t_{\alpha}-s_{\alpha})^2$$
for all $(s,t,x,y)$. By taking $y=y_{\alpha}$, $t=t_{\alpha}$ we
get that the function $(s,x)\mapsto u(s,x)-\alpha
f(d(x,y_\alpha)^2)-\alpha(t_{\alpha}-s)^2$ has a maximum at
$(s_{\alpha}, x_{\alpha})$, which (bearing in mind that
$f'(0)=0=f''(0)$) yields
$$\left(-2\alpha(t_{\alpha}-s_{\alpha}), \, 0, 0 \right)\in\mathcal{P}^{\, 2,
+}u(s_\alpha,x_{\alpha}).$$ Similarly, we also deduce that
$$
\left(-2\alpha(t_{\alpha}-s_{\alpha}), \, 0, \, 0
\right)\in\mathcal{P}^{\, 2, -}v(t_\alpha,y_{\alpha}).
$$
Since $u$ is a strict subsolution and $v$ is a supersolution, we get
    $$
    -2\alpha (t_{\alpha}-s_{\alpha})\leq \frac{-\varepsilon}{T^2}<0\leq -2\alpha (t_{\alpha}-s_{\alpha}),
    $$
a contradiction.
\end{proof}

\medskip
The preceding proof can be easily modified to yield the following
more general results.

\begin{rem}
{\em  One can replace the compactness of $M$ in the statement of
Theorem \ref{compact} by the following condition on the behavior of
$u$ and $v$ at $\infty$:
\begin{equation}\label{condition on infinity}
\limsup_{(t,x)\to\infty} u(t,x)-v(t,x)\leq 0
\end{equation}
(this condition is meant to be empty when $M$ is compact).}
\end{rem}
In the case when $M$ does not have positive curvature, one can prove
the following.
\begin{thm}\label{general comparison result}
Let $M$ be a complete Riemannian manifold with sectional curvature
bounded below and positive injectivity radius. Let $F$ satisfy
conditions ({\bf A - D}) of Section 2. Assume furthermore that
there exist $f\in\mathcal{F}(F)$ and $C>0$ such that
\begin{equation} \label{condition on f}
tf'(t)\leq C f(t) \textrm{ for all } t>0, \end{equation} and that
$F$ satisfies the following uniform continuity assumption with
respect to the variable $D^2 u$:
\begin{equation}\label{condition on uniform continuity of F}
F(\zeta, P-\delta I)-F(\zeta, P)\xrightarrow{\delta \to 0} 0
\textrm{ uniformly on } \zeta, P.
\end{equation}
Let $u\in USC([0,T)\times M)$ be a subsolution and $v\in
LSC([0,T)\times M)$ be a supersolution of
\eqref{curvatureevolutioneq} on $M$. Suppose that $u \le v$ on
$\{0\}\times M$ and \begin{equation} \limsup_{(t,x)\to\infty}
u(t,x)-v(t,x)\leq 0.
\end{equation} Then $u \le v$ on $[0,T)\times M$.
\end{thm}
\begin{proof}
Assume that the sectional curvature of $M$ is bounded below by
$-K_0$, with $K_{0}>0$. We have, with the notation used in the
proof of Theorem \ref{compact}, case 1, following equation
\eqref{estimation of A}, that $\lambda_{n}>0$, and
$$\lambda_{i}+\varepsilon_{\alpha}\lambda_{i}^{2}\leq
2\lambda_{n}\leq 4\alpha K_0  \Psi(x_{\alpha},y_{\alpha})
f'(\Psi(x_{\alpha},y_{\alpha})),$$ hence
$$    \left(A_{\alpha}+\varepsilon_{\alpha}A_{\alpha}^{2}\right)
(v,L_{x_\alpha y_\alpha}v)^{2}\leq 4\alpha K_0
\Psi(x_{\alpha},y_{\alpha}) f'(\Psi(x_{\alpha},y_{\alpha}))
||v||^2.$$
 Thus, inequality \eqref{ineqcuadratic} and condition \eqref{condition on f} imply that
 \begin{align}\label{uniform continuity of F}
 P_\alpha(v)^2-L_{y_\alpha x_\alpha}(Q_\alpha)(v)^2&=
 P_\alpha(v)^2-Q_\alpha(L_{x_\alpha y_\alpha}v)^2 \notag \\ &\le
   \left(A_{\alpha}+\varepsilon_{\alpha}A_{\alpha}^{2}\right)(v,L_{x_\alpha
   y_\alpha}v)^{2} \notag \\ &
\le 4\alpha K_0 \Psi(x_{\alpha},y_{\alpha})
f'(\Psi(x_{\alpha},y_{\alpha})) ||v||^2 \notag \\
& \leq 4\alpha K_0 C f(\Psi(x_{\alpha},y_{\alpha})) ||v||^2.
\end{align} Let us denote
$$\delta_\alpha:=4\alpha K_0 C f(\Psi(x_{\alpha},y_{\alpha})).$$ We have that $\lim_{\alpha \to
\infty}\delta_\alpha=0$. From \eqref{uniform continuity of F} we
obtain $P_\alpha -\delta_\alpha I \le
 L_{y_\alpha x_\alpha} (Q_\alpha)$.
  Then, equation \eqref{Funder-Fupper},
  the fact that $F$ is elliptic, and condition \eqref{condition on uniform continuity of
  F}
 imply that
\begin{align*}
  0<c&\le  F( \zeta_\alpha,L_{y_\alpha \, x_\alpha}(Q_\alpha))-F(\zeta_\alpha, P_\alpha) \\
    &\le F( \zeta_\alpha,P_\alpha -\delta_\alpha I)-F(\zeta_\alpha, P_\alpha)
    \xrightarrow{\alpha \to \infty} 0, \end{align*}
which again leaves us with a contradiction. The proof of case 2
parallels that in Theorem \ref{compact}.
\end{proof}

\begin{rem}
{\em Condition \eqref{condition on f} is always met when one is able
to take an $f$ of the form $$f(t)=t^{k},$$ with $k\geq 2$.
Therefore, in the cases when $F$ is given by the evolutions by mean
curvature or by Gaussian curvature, \eqref{condition on f} is
automatically satisfied.

On the other hand, condition \eqref{condition on uniform continuity
of F} is also clearly met by the function $F$ associated to the mean
curvature evolution problem. Indeed, in this case the function
$A\mapsto F(\zeta, A)$ is linear, so we have
$$
F(\zeta, P-\delta I)-F(\zeta, P)=-\delta F(\zeta, I)=\delta\,
\trace{\left(I-\frac{\zeta\otimes\zeta}{|\zeta|^{2}}\right)}\leq
\delta (n-1),
$$
where $n$ is the dimension of $M$. We thus recover Ilmanen's Theorem
from \cite{Ilmanen}:}
\end{rem}
\begin{cor}[Ilmanen]\label{Comparison for mean curvature}
Let $M$ be complete, with sectional curvature bounded below and
positive injectivity radius. Let $F$ be given by
\eqref{meancurvature2}. Let $u\in USC([0,T)\times M)$ be a
subsolution and $v\in LSC([0,T)\times M)$ be a supersolution of
\eqref{meancurvature} on $M$. Suppose that $u \le v$ on
$\{0\}\times M$ and $\limsup_{(t,x)\to\infty} u(t,x)-v(t,x)\leq
0$. Then $u \le v$ on $[0,T)\times M$.
\end{cor}

Unfortunately, condition \eqref{condition on uniform continuity of
F} in Theorem \ref{general comparison result} is not satisfied by
the function $F$ given by \eqref{F for Gaussian curvature}
corresponding to the evolution of level sets by Gaussian curvature.
In this case, we can only apply Theorem \ref{compact} in order to
deduce a comparison result for manifolds of nonnegative curvature:
\begin{cor}\label{comparison for Gaussian curvature}
Let $M$ be a complete Riemannian manifold of nonnegative sectional
curvature and positive injectivity radius. Let $F$ be given by
\eqref{F for Gaussian curvature}. Let $u\in USC([0,T)\times M)$ be
a subsolution and $v\in LSC([0,T)\times M)$ be a supersolution of
\eqref{positiveGaussiancurvature} on $M$. Suppose that $u \le v$
on $\{0\}\times M$ and $\limsup_{(t,x)\to\infty} u(t,x)-v(t,x)\leq
0$. Then $u \le v$ on $[0,T)\times M$.
\end{cor}

Given the form of the equation \eqref{curvatureevolutioneq}, it
immediately follows that, in all cases where comparison holds, one
has continuous dependence of solutions with respect to initial
data.

\begin{rem}\label{continuous dependence of solutions with respect to initial
data} If $u, v$ are solutions with initial conditions $g$ and $h$
respectively, and $\|g-h\|_{L^{\infty}(M)}\leq\varepsilon$, then
$\|u-v\|_{L^{\infty}(M\times [0, T))}\leq\varepsilon$.
\end{rem}

\bigskip

\section{Existence by Perron's Method}

We will have to use the following estimation for the second
derivative of the distance to a fixed point.
\begin{lem}\cite[p. 153]{Sakai}\label{Sakais lemma}
Let $M$ be a complete Riemannian manifold whose sectional curvature
$K$ satisfies $\delta\leq K\leq\Delta$. Suppose
$0<r<\min\{i_{M}(x_{0}), \pi/2\sqrt{\Delta}\}$. Then, for all $x\in
B(x_{0}, r)$ and $v\bot\nabla d(\cdot, x_{0})(x)$, one has
$$
\frac{c_{\Delta}(d(x,x_{0}))}{s_{\Delta}(d(x,x_{0}))}\langle v,
v\rangle \leq D^{2}d(\cdot, x_{0})(x)(v,v)\leq
\frac{c_{\delta}(d(x,x_{0}))}{s_{\delta}(d(x,x_{0}))}\langle v,
v\rangle,
$$
and the gradient $\nabla d(\cdot, x_{0})(x)$ belongs to the null
space of $D^{2}d(\cdot, x_{0})(x)$.
\end{lem}
Here $s_{\delta}$ and $c_{\delta}$ are defined by $$
s_{\delta}(t):=\left\{  \begin{array}{ll}
                                                       (\sin(\sqrt{\delta}\, t))/\sqrt{\delta}, & \hbox{ $\delta >0$;} \\
                                                       t, & \hbox{$\delta =0$;} \\
                                                       (\sinh(\sqrt{|\delta|}\, t))/\sqrt{\delta}, & \hbox{ $\delta <0$,}
                                                     \end{array}
                                                   \right.$$ and $$
c_{\delta}(t):=\left\{ \begin{array}{ll} \cos(\sqrt{\delta}\, t), & \hbox{$\delta>0$;} \\
                                                       1, & \hbox{$\delta =0$;} \\
                                                       \cosh(\sqrt{|\delta|}\, t), & \hbox{$\delta<0$.}
                                                     \end{array}
                                                   \right. $$
Notice that \begin{equation}\label{limit of tc(t)/s(t)} \lim_{t\to
0}\frac{t\, c_{\Delta}(t)}{s_{\Delta}(t)}=1.\end{equation}

\begin{prop}\label{closed under sup}
Let $F$ satisfy conditions ({\bf A - D}) of Section 2, and assume
$\mathcal{F}(F)\neq\emptyset$. Let $\mathcal{S}$ be a nonempty
family of subsolutions of
\begin{equation}\label{eqPerron}
u_{t}+F(Du, D^2 u)=0,
\end{equation}
and define $$W(z):=\sup\{v(z) : v\in\mathcal{S} \}.$$ Suppose that
$W^{*}(z)<+\infty$ for all $z\in [0,T)\times M$. Then $W^{*}$ is a
subsolution of \eqref{eqPerron} on $[0, T)\times M$.
\end{prop}
\begin{proof}
Let $\varphi\in\mathcal{A}(F)$ be such that $W^{*}-\varphi$ has a
strict maximum at $z_{0}=(t_{0}, x_{0})$. We may assume that
$W^{*}(z_{0})-\varphi(z_{0})=0$.

\noindent{\bf Case 1.} Suppose first that $D\varphi(z_{0})\neq 0$,
and let us see that $$\varphi_{t}(z_{0})+F(D\varphi(z_{0}), D^2
\varphi(z_{0}))\leq 0.$$ Define
$\psi(t,x):=\varphi(t,x)+f(d(x,x_{0}))+(t-t_{0})^4$, where
$f\in\mathcal{F}(F)$, and observe that
\begin{equation}\label{subperron1}
W^{*}(z)-\psi(z)\leq -f(d(x, x_{0}))-(t-t_{0})^4.\end{equation} Also
notice that $\psi\in\mathcal{A}(F)$,
$\psi_{t}(z_{0})=\varphi_{t}(z_{0})$,
$D\psi(z_{0})=D\varphi(z_{0})$, and
$D^2\psi(z_{0})=D^2\varphi(z_{0})$.

By definition of $W^{*}$ there exist $z'_{k}$ such that
$\lim_{k\to\infty}z'_{k}=z_{0}$ and
$$\alpha_{k}:=W^{*}(z'_{k})-\psi(z'_{k})\to
W^{*}(z_0)-\psi(z_0)=0.$$ Now, by definition of $W$, there exists a
sequence $(v_{k})\subset\mathcal{S}$ such that
$v_{k}(z'_{k})>W(z'_{k})-\frac{1}{k}$, which implies
\begin{equation}\label{subperron2}
(v_{k}-\psi)(z'_{k})>a_{k}-\frac{1}{k}.
\end{equation}
Since $v_{k}\leq W$, \eqref{subperron1} implies
\begin{equation}\label{subperron3}
(v_{k}-\psi)(z)\leq -f(d(x, x_{0}))-(t-t_{0})^4 \textrm{ for all }
z.
\end{equation}
Let $B$ be a closed ball of center $z_{0}$. Since $v_{k}-\psi$ is
upper semicontinuous it attains its maximum on $B$ at some point
$z_{k}\in B$. From \eqref{subperron2} and \eqref{subperron3} we get
$$\alpha_{k}-\frac{1}{k}< (v_{k}-\psi)(z'_{k})\leq
(v_{k}-\psi)(z_{k})\leq -f(d(x_{k},x_{0}))-(t_{k}-t_{0})^4\leq 0,$$
and since $\alpha_{k}\to 0$ we deduce that $z_{k}\to z_{0}$ and
$t_{k}\to t_{0}$. Moreover, $v_{k}-\psi$ has a local maximum at
$z_{k}$.

Since $D\psi(z_{0})=D\varphi(z_{0})\neq 0$, we have $D\psi(z)\neq 0$
for all $z$ in a neighborhood (which we may assume to be $B$) of
$z_{0}$. Because $v_{k}$ is a subsolution and $D\varphi(z_{k})\neq
0$, we get
$$
\psi_{t}(z_{k})+F(D\psi(z_{k}), D^2 \psi(z_{k}))\leq 0.
$$
Therefore, by taking limits and using the continuity of $F$ off
$\{\zeta =0\}$ and the continuity of $\psi_t$, $D\psi$, $D^2\psi$,
we obtain
$$
\varphi_{t}(z_{0})+F(D\varphi(z_{0}), D^2 \varphi(z_{0}))=
\psi_{t}(z_{0})+F(D\psi(z_{0}), D^2 \psi(z_{0}))\leq 0,
$$
and we conclude that $W^{*}$ is a subsolution of \eqref{eqPerron} at
$z_{0}$.

\medskip

\noindent{\bf Case 2.} Assume now that $D\varphi(z_{0})= 0$, and
let us check that $\varphi_{t}(z_{0})\leq 0$. Since
$\varphi\in\mathcal{A}(F)$, there exist $\delta_{0}>0$, $\omega\in
C(\mathbb{R})$ with $\omega(r)=o(r)$, and $f\in\mathcal{F}(F)$
such that
\begin{equation}\label{subperron4}|\varphi(x,t)-\varphi(z_{0})-\varphi_{t}(z_{0})(t-t_{0})|\leq
f(d(x,x_{0}))+\omega(t-t_{0})\end{equation} for all $z=(t,x)\in
B:=B(z_{0}, \delta_{0})$. We may assume that $ \omega\in
C^1(\mathbb{R})$, $\omega(0)=0=\omega'(0)$, and $\omega(r)>0$ for
$r>0$. Let us define
$$
\psi(t,x):=\varphi_{t}(z_{0})(t-t_{0})+2f(d(x,
x_{0}))+2\omega(t-t_{0}), \textrm{ and }$$
$$
\psi_{k}(t,x):=\varphi_{t}(z_{0})(t-t_{0})+2f(d(x,
x_{0}))+2\omega_{k}(t-t_{0}),$$ where $(\omega_{k})$ is a sequence
of $C^2$ functions on $\mathbb{R}$ such that $\omega_{k}\to
\omega$ and $\omega'_{k}\to \omega'$ uniformly on $\mathbb{R}$.

From \eqref{subperron4} we deduce that $W^{*}-\psi$ has a local
strict maximum at $z_{0}$. On the other hand it is clear that
$(\psi_{k})\subset\mathcal{A}(F)$, and $\psi_{k}\to\psi$ uniformly.
Arguing as in Case 1, we may find a sequence of subsolutions
$(v_{k})\subset \mathcal{S}$ and a sequence of points $z_{k}$ such
that $z_{k}\to z_{0}$ and $v_{k}-\psi_{k}$ attains a maximum at
$z_{k}$. Since $v_{k}$ is a subsolution we have
\begin{equation}\label{psik is subsolution}
(\psi_{k})_{t}(z_{k})+F(D\psi_{k}(z_{k}), D^{2}\psi_{k}(z_{k}))\leq
0 \, \textrm{ for all } k, \, \textrm{ when } x_{k}\neq x_{0}, \,
\textrm{ and }
\end{equation}
\begin{equation}\label{psik is subsolution even when the derivative
vanishes} (\psi_{k})_{t}(z_{k})\leq 0, \, \textrm{ when } x_{k}=
x_{0}.
\end{equation}
Notice that \begin{equation}\label{convergence of psikt}
\lim_{k\to\infty}(\psi_{k})_{t}(z_{k})=
\varphi_{t}(z_{0}).\end{equation} If $x_{k}=x_{0}$ for infinitely
many $k$'s, we immediately deduce from \eqref{psik is subsolution
even when the derivative vanishes} and \eqref{convergence of
psikt} that $\varphi_{t}(z_{0})\leq 0$.

\medskip

Therefore we may assume that $x_{k}\neq x_{0}$ for all $k$.
If we set \begin{align*}& \zeta_{k}:=-\exp_{x_{k}}^{-1}(x_{0}),\\
& A_{k}:=D^{2}d(\cdot, x_{0})(x_{k})
\end{align*}
then we have that $|\zeta_{k}|=d(x_{k}, x_{0})$, and
\begin{align*} & (\psi_{k})_{t}(z_{k})=\varphi_{t}(z_{0})+2\omega'_{k}(t_{k}-t_{0}) \\
& D\psi_{k}(z_{k})=\frac{2}{|\zeta_{k}|} f'(|\zeta_{k}|)\, \zeta_{k}, \\
&D^{2}\psi_{k}(z_{k})=2f''(|\zeta_{k}|)\,
\frac{\zeta_{k}\otimes\zeta_{k}}{|\zeta_{k}|^{2}}+2f'(|\zeta_{k}|)A_{k}.
\end{align*}
Since $F$ is geometric we have
\begin{equation}\label{form of F} F(D\psi_{k}(z_{k}),
D^{2}\psi_{k}(z_{k}))=2f'(|\zeta_{k}|)F(\frac{\zeta_{k}}{|\zeta_{k}|},
A_{k}).\end{equation}

Next, because $B=B(z_{0}, \delta_{0})$ is compact, we may find
numbers $\Delta, \delta>0$ such that the sectional curvature $K$ of
$M$ satisfies $\delta\leq K\leq \Delta$ on $B$. We may of course
assume $\delta_{0}<\min\{i_{M}(x_{0}), \pi/2\sqrt{\Delta}\}$, so
that we can apply Lemma \ref{Sakais lemma}: we obtain that
$A_{k}(\zeta_{k},\zeta_{k})=0$, and for all $v\in TM_{x_{k}}$ such
that $v\bot\zeta_{k}$ we have
\begin{equation*}
\frac{c_{\Delta}(|\zeta_{k}|)}{s_{\Delta}(|\zeta_{k}|)}\langle v,
v\rangle \leq A_{k}(v,v)\leq
\frac{c_{\delta}(|\zeta_{k}|)}{s_{\delta}(|\zeta_{k}|)}\langle v,
v\rangle.
\end{equation*}
This implies \begin{equation}\label{estimation for Ak}
\frac{c_{\Delta}(|\zeta_{k}|)}{s_{\Delta}(|\zeta_{k}|)}
\left(I-\frac{\zeta_{k}\otimes\zeta_{k}}{|\zeta_{k}|^{2}}\right)\leq
A_{k}\leq \frac{c_{\delta}(|\zeta_{k}|)}{s_{\delta}(|\zeta_{k}|)}
\left(I-\frac{\zeta_{k}\otimes\zeta_{k}}{|\zeta_{k}|^{2}}\right).
\end{equation}
On the other hand, equation \eqref{limit of tc(t)/s(t)} tells us
that
$$\frac{c_{\Delta}(t)}{s_{\Delta}(t)} \geq\frac{1}{2t} \, \textrm{ and }
\frac{c_{\delta}(t)}{s_{\delta}(t)}\leq \frac{2}{t}
$$ if $t>0$ is
small enough. Hence we have
$$\frac{c_{\Delta}(|\zeta_{k}|)}{s_{\Delta}(|\zeta_{k}|)}
\geq\frac{1}{2|\zeta_{k}|} \, \textrm{ and }
\frac{c_{\delta}(|\zeta_{k}|)}{s_{\delta}(|\zeta_{k}|)} \leq
\frac{2}{|\zeta_{k}|}
$$
for $k$ large enough, which we may assume are all $k$. By plugging
these inequalities into \eqref{estimation for Ak} we obtain
\begin{equation}\label{estimation for Ak2}
\frac{1}{2|\zeta_{k}|}\left(I-\frac{\zeta_{k}\otimes\zeta_{k}}{|\zeta_{k}|^{2}}\right)\leq
A_{k}\leq
\frac{2}{|\zeta_{k}|}\left(I-\frac{\zeta_{k}\otimes\zeta_{k}}{|\zeta_{k}|^{2}}\right)
\end{equation}
Bearing in mind that $F$ is elliptic and geometric, we get
\begin{eqnarray*}\label{estimation for F(zetak Ak)}
& & \frac{1}{|\zeta_{k}|}F(\zeta_{k}, 2I)=
F(\frac{\zeta_{k}}{|\zeta_{k}|}, \frac{2}{|\zeta_{k}|} I
)=F\left(\frac{\zeta_{k}}{|\zeta_{k}|},
\frac{2}{|\zeta_{k}|}\left(I-\frac{\zeta_{k}\otimes\zeta_{k}}{|\zeta_{k}|^{2}}\right)
\right)\leq \\ & &F(\frac{\zeta_{k}}{|\zeta_{k}|}, A_k)\leq \\
& &F\left(\frac{\zeta_{k}}{|\zeta_{k}|},
\frac{1}{2|\zeta_{k}|}\left(I-\frac{\zeta_{k}\otimes\zeta_{k}}{|\zeta_{k}|^{2}}\right)
\right)=F(\frac{\zeta_{k}}{|\zeta_{k}|},\frac{1}{2|\zeta_{k}|}
\,I)\leq \frac{1}{|\zeta_{k}|}F(\zeta_{k}, -2I),
\end{eqnarray*}
which combined with \eqref{form of F} yields \begin{equation}
2\frac{f'(|\zeta_{k}|)}{|\zeta_{k}|}F(\zeta_{k}, 2I)\leq
F(D\psi_{k}(z_{k}), D^{2}\psi_{k}(z_{k}))\leq
2\frac{f'(|\zeta_{k}|)}{|\zeta_{k}|}F(\zeta_{k}, -2I),
\end{equation}
which, thanks to condition \eqref{limit property of f, F}, allows to
conclude that \begin{equation}\label{Fofpsi goes to
0}\lim_{k\to\infty}F(D\psi_{k}(z_{k}),
D^{2}\psi_{k}(z_{k}))=0.\end{equation} Finally, from \eqref{psik is
subsolution}, \eqref{convergence of psikt} and \eqref{Fofpsi goes to
0}, it follows that $$\varphi_{t}(z_{0})\leq 0.$$ In either case we
see that $W^{*}$ is a subsolution of \eqref{eqPerron}.
\end{proof}

\medskip

\begin{thm}\label{Existence}
Let $F$ satisfy conditions ({\bf A - D}) of Section 2, and assume
that $\mathcal{F}(F)\neq\emptyset$ and comparison holds for the
equation
\begin{equation}\label{equation for existence}
\left\{
  \begin{array}{ll}
    u_{t}+F(Du, D^2u)=0\\
    u(0,x)=g(x).
  \end{array}
\right.
\end{equation} Let $\underline{u}$ and $\overline{u}$ be a
subsolution and a supersolution of \eqref{equation for existence},
respectively, satisfying
$\underline{u}_*(0,x)=\overline{u}^*(0,x)=g(x)$. Then $w=\sup \{ v :
\underline{u} \leq v \leq \overline{u}, \,\,
   v  \text{ is a subsolution}\}$ is a solution of \eqref{equation for existence}.
\end{thm}
\begin{proof}
From $=\underline{u}_*\leq w_*\leq w \leq w^*\leq \overline{u}^*$,
we deduce that $w_*(0,x)=w(0,x)=w^*(0,x)=g(x)$. On the other hand
$w^*$ is a subsolution by Proposition \ref{closed under sup}, and
$w^*\leq \overline{u}$ by comparison, hence $w^*=w$ by definition of
$w$. We claim that $w_*$ is a supersolution. This implies $w^*\leq
w_*$ by comparison, hence $w_*=w=w^*$ and consequently $w$ is a
solution.

Let us prove the claim. Suppose to the contrary that $w_{*}$ is
not a supersolution. Then there exist $z_{0}=(t_{0}, x_{0})$ and a
$C^2$ function $\varphi$ such that $(w_{*}-\varphi)(z)\geq
0=(w_{*}-\varphi)(z_{0})$ for all $z$, and either
\begin{equation}\label{first case of existence}
\varphi_{t}(z_{0})+F(D\varphi(z_{0}), D^{2}\varphi(z_{0}))< 0, \,
\textrm{ when } D\varphi(z_{0})\neq 0, \textrm{ or}
\end{equation}
\begin{equation}\label{second case of existence}
\varphi_{t}(z_{0})< 0, \, \textrm{ when } D\varphi(z_{0})=0.
\end{equation}

By replacing $\varphi(t,x)$ with the function
$\varphi(t,x)+d(x,x_{0})^4+(t-t_{0})^4$ on a neighborhood of $z_0$
we can furthermore assume that
    \begin{equation}\label{very strict minimum condition}
    (w_{*}-\varphi)(t,x)\geq d(x,x_{0})^4+(t-t_{0})^4.
    \end{equation}
Let us denote
    $$
    U_{\delta}:=\{(t,x): d(x,x_{0})^4+(t-t_{0})^4\leq \delta^4\}.
    $$

\noindent{\bf Case 1.} In the case when \eqref{first case of
existence} holds, by continuity of $\varphi_{t}$, $D\varphi$, $D^2
\varphi$ and $F$, we can find $r>0$ such that
$$\varphi_{t}(z)+F(D\varphi(z), D^{2}\varphi(z))< 0$$ for all $z\in
U_{2r}$, that is $\varphi$ is a subsolution on $U_{2r}$, and
obviously the same is true of
$\widetilde{\varphi}:=\varphi+r^{4}/2$.

From \eqref{very strict minimum condition} we have that
    \begin{equation}\label{bound on crown}
    w(z)\geq
    w_{*}(z)-\frac{r^{4}}{2}\geq\varphi(z)+\frac{r^{4}}{2}\,
    \textrm{ for all } z\in U_{2r}\setminus U_{r}.
    \end{equation}

Now let us define $$W(z)=\left\{
                        \begin{array}{ll}
                          \max \{ \widetilde{\varphi}(z), w(z) \}, & \hbox{ if $z\in U_r$;} \\
                          w(z), & \hbox{ otherwise.}
                        \end{array}
                      \right.
$$
By using Proposition \ref{closed under sup} and equation
\eqref{bound on crown}, it is immediately checked that $W$ is a
subsolution. We have $W=w$ outside $B(z_{0}, r)$, but
\begin{equation}\label{first contradiction for existence}
\sup (W-w)>0 \end{equation} because, by definition of $w_*$, there
exits a sequence $\{ (t_n,x_n)\}$ converging to $(t_0,x_0)$ such that
$\lim w(t_n,x_n)=w_*(t_0,x_0)$, and consequently we have
$$\lim (W(t_n,x_n)-w(t_n,x_n))\geq \lim (\widetilde{\varphi}(t_n,x_n)-w(t_n,x_n))=
r^{4}/2 >0.$$ On the other hand, $W(0,x)=w(0,x)=g(x)$, because we
could of course have taken $r>0$ small enough such that
$(0,x)\not\in B(z_{0}, r)$. We deduce (from comparison again) that
$W\leq \overline{u}$ and consequently $W\leq w$, which contradicts
\eqref{first contradiction for existence}.

\medskip

\noindent{\bf Case 2.} On the other hand, in the case when
\eqref{second case of existence} holds, since
$\varphi\in\mathcal{A}(F)$, there exist $\delta_{0}>0$, $\omega\in
C(\mathbb{R})$ with $\omega(r)=o(r)$, and $f\in\mathcal{F}(F)$ such
that
\begin{equation*}|\varphi(x,t)-\varphi(z_{0},
t_{0})-\varphi_{t}(t_{0},x_{0})(t-t_{0})|\leq
f(d(x,x_{0}))+\omega(t-t_{0})\end{equation*} for all $z=(t,x)\in
B:=B(z_{0}, \delta_{0})$. We may assume that $ \omega\in
C^1(\mathbb{R})$, $\omega(0)=0=\omega'(0)$, and $\omega(r)>0$ for
$r>0$. Let us define
$$
\psi(t,x)=\varphi(z_{0})+ \varphi_{t}(z_{0})(t-t_{0})-2f(d(x,
x_{0}))-2\omega(t-t_{0}).$$ Then $w_{*}-\psi$ attains a strict
minimum at $z_{0}$. Also notice that $D\psi(z)\neq 0$ for $z\neq
z_{0}$. Arguing as in Case 2 of the proof of Proposition
\ref{closed under sup}, one can show that
\begin{equation}\label{lim property of F the second}
\lim_{z\to z_{0}} F(D\psi(z), D^{2}\psi(z))=0.
\end{equation}
By combining this with the continuity of $\psi_{t}$ and the fact
that $\psi_{t}(z_{0})=\varphi_{t}(z_{0})<0$, we can find an $r>0$
such that
$$\psi_{t}(z)+F(D\psi(z), D^{2}\psi(z))< 0$$ for all $z\in
U_{2r}, z\neq z_{0}$. The rest of the proof is identical to that
of Case 1 (just replace $\varphi$ with $\psi$).
\end{proof}

\bigskip

Let us now show how to apply the above Theorem in order to
construct solutions of \eqref{equation for existence}. We will
need to use the following stability result.
\begin{lem}\label{sequences of solutions} Assume that u.s.c. (respectively l.s.c.) functions $u_k$
are subsolutions (supersolutions, respectively) of
\eqref{curvatureevolutioneq}. Assume also that $\{ u_k\}$
converges locally uniformly to a function $u$. Then $u$ is
subsolution (supersolution, respectively) of
\eqref{curvatureevolutioneq}.
\end{lem}
\begin{proof}
Suppose that $\varphi\in \mathcal{A}(F)$ and $u-\varphi  \textrm{
attains a strict local maximum at } (t_0,x_0)$. The convergence of
the subsolutions $u_{k}$ allows us to find a sequence of local
maxima $(t_k,x_k)$ of $u_k-\varphi$ which converges to
$(t_0,x_0)$. Then, by a similar argument to that of the proof of
Proposition \ref{closed under sup}, one can show that $u$ is a
subsolution of \eqref{curvatureevolutioneq} at $(t_{0}, x_{0})$.
\end{proof}

\bigskip

Let $M$ be a compact Riemannian manifold. Assume that comparison
holds for the equation \eqref{equation for existence}. Let us
first produce solutions of \eqref{equation for existence} for
initial data $g$ in the class $\mathcal{A}(F)$.

Let us define $$\underline{u}(t,x)=-Kt +g(x), \, \textrm{ and } \,
\overline{u}(t,x)=Kt +g(x),$$ where $K:=\sup_{x\in M}|F(Dg(x),
D^{2}g(x))|$ (which is finite because $g\in \mathcal{A}(F)$ and
$M$ is compact). It is immediately seen that $\underline{u}$ is a
subsolution and $\overline{u}$ is a supersolution of
\eqref{equation for existence}, and obviously
$\underline{u}_*(0,x)=\overline{u}^*(0,x)=g(x)$. According to
Theorem 5.3 and comparison, there exists a unique solution $u$ of
\eqref{equation for existence}.

Now take $g$ a continuous function on $M$. According to
Proposition \ref{AF is dense}, we can find a sequence $g_{k}$ of
functions in $\mathcal{A}(F)$ such that $g_{k}\to g$ uniformly on
$M$. Let $u_{k}$ be the unique solution of \eqref{equation for
existence} with initial datum $g_{k}$. By Remark \ref{continuous
dependence of solutions with respect to initial data},  $(u_{k})$
is a Cauchy sequence in $\mathcal{C}([0,\infty)\times M)$, hence
it converges to some $u\in \mathcal{C}([0,\infty)\times M)$
uniformly on $[0,\infty)\times M$. Then by Lemma 6.1 it follows
that $u$ is a solution with initial datum $u(0,x)=g(x)$.

Therefore we can combine this argument with Theorems \ref{compact}
and \ref{general comparison result} to obtain the following
corollaries.

\begin{cor}
Let $M$ be a compact Riemannian manifold of nonnegative sectional
curvature, $g:M\to\mathbb{R}$ a continuous function, and let
$F:J_{0}^{2}(M) \rightarrow \mathbb R$ be continuous, elliptic,
translation invariant and geometric. Then there exists a unique
solution of \eqref{curvatureevolutioneq}  on $[0, \infty)\times
M$.
\end{cor}

\begin{cor}
Let $M$ be a compact Riemannian manifold, $g:M\to\mathbb{R}$ a
continuous function, and let $F$ satisfy conditions ({\bf A - D})
of Section 2. Assume furthermore that \eqref{condition on f} and
\eqref{condition on uniform continuity of F} are satisfied. Then
there exists a unique solution of \eqref{curvatureevolutioneq} on
$[0,\infty)\times M$.
\end{cor}

\begin{cor}[Ilmanen]
Let $M$ be a compact Riemannian manifold, $g:M\to\mathbb{R}$
continuous. There exists a unique solution $u$ of the mean
curvature evolution equation \eqref{meancurvature} such that
$u(0,x)=g(x)$.
\end{cor}

\begin{cor}
Let $M$ be a compact Riemannian manifold, $g:M\to\mathbb{R}$
continuous. There exists a unique solution $u$ of the positive
Gaussian curvature evolution equation
\eqref{positiveGaussiancurvature} such that $u(0,x)=g(x)$.
\end{cor}

\begin{cor}\label{generalized motion by mean curvature in arbitrary codimension}
Let $M$ be a compact Riemannian manifold, $g:M\to\mathbb{R}$
continuous. There exists a unique solution $u$ of the mean
curvature evolution equation in arbitrary codimension (given in
Example \ref{mean curvature in arbitrary codimension}) such that
$u(0,x)=g(x)$.
\end{cor}

When $M$ is not compact, analogous corollaries can be established
if one additionally demands that the initial datum $g$ be a
(positive) constant outside some bounded set of $M$, and that
$i(M)>0$. The proof is similar (replacing uniform convergence on
$M$ with uniform convergence on compact subsets of $M$).

\bigskip

\section{Geometric consistency \& Level set method}

\begin{thm}\label{invariance} Let $\theta:\mathbb{R}\rightarrow \mathbb{R}$
be a continuous function, and let $u$ be a bounded continuous
solution of \eqref{curvatureevolutioneq}. Then $v=\theta\circ u$
is also a solution. Moreover, if $\theta$ is nondecreasing and $u$
is a subsolution (resp. supersolution) then $v=\theta\circ u$ is a
subsolution (resp. supersolution) as well.
\end{thm}
\begin{proof}
Assume first that $\theta$ is monotone. We may consider a sequence
of smooth functions $\theta_k$ with nonvanishing derivatives,
converging uniformly to $\theta$ over the bounded range of $u$.
Hence by Lemma \ref{sequences of solutions}, we may directly
assume that $\theta'\neq 0$. Notice that $g=\theta^{-1}$ satisfies
$g'\neq0$ too.

\medskip

Suppose first that $\theta'>0$. Let $\varphi\in\mathcal{A}(F)$  and
assume that $\theta\circ u - \varphi$ attains a local maximum at
$z_0$. If we denote $\psi =g \circ \varphi$, it is not difficult to
check that $\psi\in\mathcal{A}(F)$, and $u-\psi$ clearly attains a
local maximum at $(t_0,x_0)$. Consequently
$$\psi_t(z_0)+F(D\psi(z_{0}), D^{2}\psi(z_{0})) \leq 0$$ if $D\psi (t_0,x_0) \not= 0$, and $
\psi_t(t_0,x_0) \leq 0$ otherwise. But $D\psi (z_0)= 0$ if and only
if $D\varphi (z_0)= 0$, and
\begin{eqnarray*} & &\psi_{t}=g'(\varphi)\varphi_{t} \\
                            & & D\psi=g'(\varphi) D\varphi \\
                             & & D^{2}\psi=g''(\varphi)
 D\psi\otimes D\psi+g'(\varphi) D^{2}\varphi.
\end{eqnarray*}
Since $F$ is geometric and $g'>0$, one immediately sees that
$$ \varphi_t(z_0)+F(D\varphi(z_{0}), D^{2}\varphi(z_{0}))=
\frac{1}{g'(\varphi)(z_{0})}\left( \psi_t(z_0)+F(D\psi(z_{0}),
D^{2}\psi(z_{0}))  \right)\leq 0$$ if $D\varphi(z_0) \not= 0$, and $
\varphi_t(t_0,x_0) =\frac{1}{g'(\varphi(z_{0}))}\psi_{t}(z_{0})\leq
0$ otherwise. This shows that $\theta\circ u$ is a subsolution.

If $\theta'<0$, the same argument tells us that if $u$ is
subsolution (respectively supersolution), then $v$ is
supersolution (respectively subsolution). In order to establish
the result for continuous functions, it is enough to observe that
a continuous function can be uniformly approximated by a sequence
of locally monotone functions. Then a local application of Lemma
\ref{sequences of solutions} yields the result.
\end{proof}

\bigskip

Now one can show that, if comparison and existence hold for
\eqref{curvatureevolutioneq} (e.g. when $M$ is a compact Riemannian
manifold of nonnegative curvature), then for every compact level set
$\Gamma_{0}$ there is a unique, well-defined, level set evolution
$\Gamma_{t}$ of $\Gamma_{0}$ by the geometric curvature evolution
equation corresponding to \eqref{curvatureevolutioneq}.

Let $g$ be a continuous function on $M$ with $\Gamma_{0}=\{ x\in M :
g(x)=0\}$, and assume that $\Gamma_0$ is compact. We may also assume
that $g$ is constant outside a bounded neighborhood of $\Gamma _0$,
and in particular bounded. Let $u$ be the unique solution of
\eqref{curvatureevolutioneq} with $u(0,\cdot)=g$. We define
$$\Gamma _t=\{ x\in M: u(t,x)=0\}.$$
\begin{thm}\label{geometric character}
Assume that comparison and existence hold for
\eqref{curvatureevolutioneq}. Let $\hat{g}:M \rightarrow
\mathbb{R}$ be a continuous function satisfying $\Gamma _0=\{ x\in
M : \hat{g}(x)=0\}$ and such that $\hat{g}$ is constant outside a
bounded neighborhood of $\Gamma _0$. Let $\hat{u}$ be the unique
continuous solution of \eqref{curvatureevolutioneq} with initial
condition $\hat{g}$. Then
$$\Gamma _t=\{ x\in M : \hat{u}(t,x)=0\}.$$
\end{thm}
\begin{proof} This is a consequence of Theorem \ref{invariance} and the comparison
principle. It follows exactly as in the case $M=\mathbb{R}^{n}$, see
\cite[Theorem 5.1]{ES1}, or \cite[Chapter 4]{Giga}, for instance.
\end{proof}

\begin{cor} The definition of $\Gamma _t=\{ x\in M : u(t,x)=0\}$ does not depend upon the particular choice of the function
$g$ satisfying $\Gamma _0=\{ x\in M : g(x)=0\}$.
\end{cor}

It can also be checked that the evolution $\Gamma_{0}\mapsto
\mathcal{K}(t)\Gamma_{0}:=\Gamma_{t}$ thus defined has the semigroup
property
$$
\mathcal{K}(t+s)=\mathcal{K}(t)\mathcal{K}(s).
$$

Some other properties of the evolutions can be established as in the
case $M=\mathbb{R}^{n}$. For instance, in the case of the evolution
by mean curvature, it is possible to show that if
$\Gamma_{0}=\partial U$ is a smooth connected hypersurface with
positive mean curvature with respect to the inner unit normal field,
then $\Gamma_{t}$ continues to have positive mean curvature as long
as it exists, in the sense that
$$\Gamma_{t}=\{x\in M: v(x)=t\},$$
where $v$ is the solution of the stationary problem
$$
\left\{
  \begin{array}{ll}
   \displaystyle{ -\trace\left((I-\frac{Du\otimes Du}{|Du|^{2}}) D^{2}u\right)=1}, & \hbox{ on $U$;} \\
    v=0, & \hbox{ on $\Gamma_{0}=\partial U$,}
  \end{array}
\right.
$$
(which admits a unique viscosity solution, see \cite{AFS}).

However, one has to be very cautious and not take it for granted
that all the usual geometrical properties of the generalized
evolutions by mean curvature or by Gaussian curvature could be
immediately extended from the Euclidean to the Riemannian setting.
As a matter of fact, many of these properties are very likely to
fail in the case of manifolds of negative curvature. We will
present several counterexamples and related conjectures in Section
8.

\bigskip

\section{Consistency with the classical motion}
In this section we suppose that equation
(\ref{curvatureevolutioneq}) arises from a classical geometric
evolution for hypersurfaces in $M$. We establish the consistency
of the level set evolution equation with this classical geometric
motion.

More precisely, suppose $\left(  \Gamma_{t}\right)  _{t\in\left[
0,T\right] }$ is a family of smooth, compact, orientable
hypersurfaces in $M$ evolving according to a classical geometric
motion, locally depending only on its normal vector fields and
second fundamental forms. In particular, we shall assume that
$\Gamma_{t}$ is the boundary of a bounded open set $U_{t}\subset
M$ and that there exists a family of diffeomorphisms of manifolds
with
boundary%
\[
\phi^{t}:\overline{U_{0}}\rightarrow\overline{U_{t}},\qquad
t\in\left[ 0,T\right]  ,
\]
such that:

\begin{itemize}
\item[(i)] $\phi^{0}=$\textrm{Id}, and,

\item[(ii)] for every $x\in\Gamma_{0}$ the following holds:%
\begin{equation}
\frac{d}{dt}\phi^{t}\left(  x\right)  =G\left(  \nu\left(
t,\phi^{t}\left( x\right)  \right)  ,\nabla^{\Gamma}\nu\left(
t,\phi^{t}\left(  x\right)
\right)  \right)  , \label{CHE}%
\end{equation}
where $\nu\left(  t,\cdot\right)  $ is a unit normal vector field
to $\Gamma_{t}$, and the linear map
\[
\nabla^{\Gamma}\nu\left(  t,x\right)  :\left(  T\Gamma_{t}\right)
_{x}\ni \xi\mapsto\nabla_{\xi}^{T}\nu\left(  t,x\right)  \in\left(
T\Gamma_{t}\right)
_{x}%
\]
and $\nabla^{T}$ stands for the orthogonal projection onto $\left(
T\Gamma_{t}\right) _{x}$ of covariant derivative in $M$.
\end{itemize}

Classical motion by mean curvature corresponds to taking $f\left(
\nu ,\nabla^{\Gamma}\nu\right)  =\mathrm{tr}\left(
-\nabla^{\Gamma}\nu\right) \nu$, whereas classical motion by
Gaussian curvature is defined by $f\left(
\nu,\nabla^{\Gamma}\nu\right)  =\mathrm{det}\left(  -\nabla^{\Gamma}%
\nu\right)  \nu$. The level set evolution equation induced by
(\ref{CHE}) is of the form (\ref{curvatureevolutioneq}) where $F$
is related to $G$ through formula (\ref{FandG}). As before, we
assume that $F$ is elliptic, translation invariant and geometric.
In this case we already know that $F$ is continuous, and in fact
smooth off $\{\zeta=0\}$, because $F$ is of the form (\ref{FandG})
with $G$ smooth.

Define $d:\left[  0,T\right]  \times M\rightarrow\mathbb{R}$, the
signed
distance function from $\Gamma_{t}$, as:%
\[
d\left(  t,x\right)  :=\left\{
\begin{array}
[c]{ll}%
\mathrm{dist}\left(  x,\Gamma_{t}\right)  & \text{if }x\in U_{t}\smallskip\\
-\mathrm{dist}\left(  x,\Gamma_{t}\right)  & \text{if }x\in
M\setminus U_{t}.
\end{array}
\right.
\]

\begin{lem}
\label{Lemmadsmooth}There exist constants $K,\delta>0$ such that
$d$ is smooth
in%
\[
I_{\delta}:=\left\{  \left(  t,x\right)  \in\left[  0,T\right]
\times M\;:\;\left\vert d\left(  t,x\right)  \right\vert
<\delta\right\}
\]
and%
\begin{equation}
\left\vert d_{t}+F\left(  Dd,D^{2}d\right)  \right\vert \leq
K\left\vert
d\right\vert \qquad\text{in }I_{\delta}. \label{resto}%
\end{equation}

\end{lem}
\begin{proof}
We consider geodesic normal coordinates from $\Gamma_{t}$,%
\[
\Phi\left(  t,s,x\right)  :=\exp_{x}\left(  s\nu\left(  t,x\right)
\right)  ,
\]
assuming that $\nu\left(  t,x\right)  $ points towards the
interior $U_{t}$ for
every $x\in\Gamma_{t}$. Clearly, for $s$ small enough,%
\[
d\left(  t,\Phi\left(  t,s,x\right)  \right)  =s.
\]
Given $x_{0}\in\Gamma_{0}$ there exists a neighborhood $V$ of
$x_{0}$ in $\Gamma_{0}$ and an interval $\left(  -r,r\right)  $
such that $\Phi\left( t,\cdot,\cdot\right)  $ is a diffeomorphism
from $\left(  -r,r\right) \times\phi^{t}\left(  V\right)  $ onto
its image $X_{t}:=\Phi\left(  \left( -r,r\right)
\times\phi^{t}\left(  V\right)  \right)  $ for $t\in\left[
0,T\right]  $. Note that $\Phi$ is also smooth in $t$. Denote by
$\Psi\left( t,\cdot\right)  $ the inverse of $\Phi\left(
t,\cdot,\cdot\right)  $ and
write%
\[
\Psi\left(  t,y\right)  :=\left(  \rho\left(  t,y\right)  ,X\left(
t,y\right)  \right)  \text{.}%
\]
Now for $x\in\phi^{t}\left(  V\right)  $ and $s\in\left(
-r,r\right)  $ we have $X\left(  t,\Phi\left(  t,s,x\right)
\right)  =x$ and $\rho\left( t,\Phi\left(  t,s,x\right)  \right)
=s$. Both $X$ and $\rho$ are smooth in $t$, and clearly $\rho=d$
in $ {\textstyle\bigcup\nolimits_{t\in\left[  0,T\right]  }}
\left\{ t\right\}  \times X_{t}=I_{\delta}$.

In order to prove (\ref{resto}) it suffices to note that, since,
for $x\in\Gamma_{t}$, $Dd\left(  t,x\right)  =\nu\left( t,x\right)
\neq0$,
\[
r\left(  t,x\right)  :=d_{t}+F\left(  Dd,D^{2}d\right)
\]
is a smooth function vanishing for $x\in\Gamma_{t}$. This gives
(\ref{resto}) locally; a global bound then follows by the
compactness of $\Gamma_{t}$.
\end{proof}
Next we state and prove the main result of this section.

\begin{thm}
Let $u$ be the unique viscosity solution to the level set equation
(\ref{curvatureevolutioneq}) with initial datum
$u|_{t=0}=d|_{t=0}$. Then, for every $t\in\left[  0,T\right]  $,
the zero level set of $u\left(
t,\cdot\right)  $ coincides with $\Gamma_{t}$:%
\[
\Gamma_{t}=\left\{  x\in M\;:\;u\left(  t,x\right)  =0\right\}  .
\]

\end{thm}

\begin{proof}
Define
\[
v\left(  t,x\right)  :=e^{tK}\left(  \left(  d\left(  t,x\right)
\vee0\right)  \wedge\delta/2\right)  ,
\]
where $K$ is the constant given by (\ref{resto}). We shall prove
that $v$ is a viscosity \emph{supersolution} to equation
(\ref{curvatureevolutioneq}). Clearly, $v|_{t=0}\geq
u|_{t=0}\wedge\delta/2$, and, by Theorem \ref{invariance},
$u\wedge\left( \delta/2\right) $ is a viscosity solution to
(\ref{curvatureevolutioneq}) as well. The comparison principle
(Theorem \ref{compact}) then will ensure that $v\geq u\wedge\left(
\delta/2\right)  $. In particular,
\begin{equation}
\left\{  x\in M\;:\;u\left(  t,x\right)  >0\right\}  \subseteq
U_{t},\qquad
t\in\left[  0,T\right]  .\label{sps in ut}%
\end{equation}
On the other hand, we shall prove that
\[
w\left(  t,x\right)  :=e^{-tK}\left(  \left(  d\left(  t,x\right)
\wedge0\right)  \vee\left(  -\delta/2\right)  \right)
\]
is a viscosity \emph{subsolution} to (\ref{curvatureevolutioneq}).
Now $w|_{t=0}\leq u|_{t=0}\vee\left(  -\delta/2\right)  $, and the
comparison principle will imply that $w\leq u\vee\left(
-\delta/2\right)  $. This, together with (\ref{sps in ut}) yields
\[
\left\{  x\in M\;:\;u\left(  t,x\right)  >0\right\}  =U_{t},\qquad
t\in\left[ 0,T\right]  .
\]
Now take $\varepsilon>0$ and let $u_{\varepsilon}:=u+\varepsilon$;
this is again a viscosity solution to
(\ref{curvatureevolutioneq}). It turns out that%
\[
v_{\varepsilon}\left(  t,x\right)  :=e^{tK}\left(  \left(  \left(
d\left( t,x\right)  +\varepsilon\right)  \vee0\right)
\wedge\delta/2\right)  ,
\]%
\[
w_{\varepsilon}\left(  t,x\right)  :=e^{-tK}\left(  \left(  \left(
d\left( t,x\right)  +\varepsilon\right)  \wedge0\right)
\vee\left(  -\delta/2\right) \right)  ,
\]
are respectively super and subsolutions to
(\ref{curvatureevolutioneq}) provided $\varepsilon$ is much
smaller than $\delta$ -- namely, small enough to ensure that
$v_{\varepsilon}$ and $w_{\varepsilon}$ are smooth in the regions
$0<d\left(t,x\right)+\varepsilon<\delta/2$ and
$-\delta/2<d\left(t,x\right)+\varepsilon<0$, respectively.
Applying the comparison principle as we did before we ensure that
\[
\left\{  x\in M\;:\;u\left(  t,x\right)  >-\varepsilon\right\}
=\left\{  x\in M\;:\;d\left(  t,x\right)  >-\varepsilon\right\}  .
\]
Letting $\varepsilon$ go to zero we conclude that the zero level
set of $u\left(  t,\cdot\right)  $ is precisely $\Gamma_{t}$, as
we wanted to prove.
\medskip
We now show our claim that $v$ is a supersolution to
(\ref{curvatureevolutioneq}). Start noticing that Lemma
\ref{Lemmadsmooth} ensures that, for $t\in\left[
0,T\right]  $ the following holds:%
\begin{equation}
\left\{
\begin{array}
[c]{ll}%
v_{t}\left(  t,x\right)  +F\left(  Dv\left(  t,x\right)
,D^{2}v\left( t,x\right)  \right)  \geq0, & \text{if }0<d\left(
t,x\right)  <\delta
/2,\medskip\\
v_{t}\left(  t,x\right)  \geq0, & \text{if }d\left(  t,x\right)
<0\text{ or }d\left(  t,x\right)  >\delta/2.
\end{array}
\right.  \label{sps}%
\end{equation}
Let $\left(  t_{0},x_{0}\right)  \in\left[  0,T\right]  \times M$
and $\varphi\in\mathcal{A}\left(  F\right)  $ be such that
$v-\varphi\ $has a local minimum at $\left(  t_{0},x_{0}\right)
$. Without loss of generality, we may assume that
\begin{equation}
\varphi\left(  t_{0},x_{0}\right)  =v\left(  t_{0},x_{0}\right)
,\label{eqsps}%
\end{equation}
and that
\begin{equation}
\varphi\leq v,\qquad\text{locally around }\left(
t_{0},x_{0}\right)
.\label{min}%
\end{equation}
Since the level sets $d\left(  t,x\right)  =c\in\left(
-\delta,\delta\right) $ are smooth hypersurfaces, necessarily
$d\left(  t_{0},x_{0}\right) \neq\delta/2$. If moreover $\left(
t_{0},x_{0}\right)  \notin\Gamma_{t}$, then (\ref{eqsps}),
(\ref{min}) and the smoothness of $v$ imply that, at $\left(
t_{0},x_{0}\right) $ one has $\varphi_{t}=v_{t}$ and
$D\varphi=Dv$. Therefore, using (\ref{sps}) we
conclude that:%
\begin{equation}
\left\{
\begin{array}
[c]{ll}%
\varphi_{t}\left(  t_{0},x_{0}\right)  +F\left(  D\varphi\left(  t_{0}%
,x_{0}\right)  ,D^{2}\varphi\left(  t_{0},x_{0}\right)  \right)
\geq0., &
\text{if }D\varphi\left(  t_{0},x_{0}\right)  \neq0,\medskip\\
\varphi_{t}\left(  t_{0},x_{0}\right)  \geq0, & \text{otherwise.}%
\end{array}
\right.  \label{phiissps}%
\end{equation}
Now we shall prove that the above identity also holds when $\left(
t_{0},x_{0}\right)  \in\Gamma_{t}$. Let $Q:=%
{\textstyle\bigcup\nolimits_{t\in\left[  0,T\right]  }}
\left\{  t\right\}  \times\Gamma_{t}$; this is precisely the set
of zeroes of $d$. As $\left\vert Dd\right\vert =1$ on
$I_{\delta}$, we infer that $Q$ is a smooth hypersurface of
$\left[  0,T\right]  \times M$; since (\ref{eqsps}) and
(\ref{min}) imply that $v-\varphi$ has a minimum at $\left(  t_{0}%
,x_{0}\right)  $, the following holds for every tangent vector
$\left(
\tau,\xi\right)  \in TQ_{\left(  t_{0},x_{0}\right)  }$:%
\begin{equation} \label{ker}
\varphi_{t}\left(  t_{0},x_{0}\right)  \tau+D\varphi\left(
t_{0},x_{0}\right)  (\xi) =0;%
\end{equation}
moreover, for every curve $\gamma:\left(  -1,1\right)
\rightarrow\Gamma_{t_{0}}$ with $\gamma\left(  0\right)  =x_{0}$,
and $\gamma^{\prime}\left(  0\right)
=\xi$ we have,%
\begin{equation}
\frac{d^{2}}{ds^{2}}\varphi\left(  t_{0},\gamma\left(  s\right)
\right)
|_{s=0}\leq0.\label{2der}%
\end{equation}
From identity (\ref{ker}) we deduce that%
\begin{equation}
\left(  \varphi_{t}\left(  t_{0},x_{0}\right)  ,D\varphi\left(  t_{0}%
,x_{0}\right)  \right)  =\lambda\left(  d_{t}\left(
t_{0},x_{0}\right)
,Dd\left(  t_{0},x_{0}\right)  \right)  ,\label{grad phi}%
\end{equation}
for some $\lambda\in\mathbb{R}$, whereas (\ref{2der}) merely states that:%
\[
\left\langle D^{2}\varphi\left(  t_{0},x_{0}\right)
\xi,\xi\right\rangle \leq-\left\langle \nabla\varphi\left(
t_{0},x_{0}\right)  ,\gamma^{\prime\prime }\left(  0\right)
\right\rangle .
\]
Taking into account that $\left\langle \nabla d\left(
t_{0},\gamma\left( 0\right)
\right)  ,\gamma^{\prime}\left(  0\right)  \right\rangle =0$, we obtain:%
\begin{align}
\left\langle D^{2}\varphi\left(  t_{0},x_{0}\right)
\xi,\xi\right\rangle  & \leq-\left\langle \nabla\varphi\left(
t_{0},x_{0}\right)  ,\gamma^{\prime\prime
}\left(  0\right)  \right\rangle \label{hess phi}\\
&  =-\lambda\left\langle \nabla d\left(  t_{0},x_{0}\right)
,\gamma^{\prime\prime }\left(  0\right)  \right\rangle
=\lambda\left\langle D^{2}d\left( t_{0},x_{0}\right)
\xi,\xi\right\rangle .\nonumber
\end{align}
Given a smooth curve $\eta:\left(  -1,1\right)  \rightarrow Q$
such that $\eta\left(  0\right)  =\left(  t_{0},x_{0}\right)  $
and $\eta^{\prime }\left(  0\right)  =-\left(  d_{t}\left(
t_{0},x_{0}\right)  ,\nabla d\left(
t_{0},x_{0}\right)  \right)  $ the following holds:%
\[
\frac{d}{dt}\varphi\left(  \eta\left(  t\right)  \right)  |_{t=0}%
=-\lambda\left(  d_{t}\left(  t_{0},x_{0}\right)  ^{2}+\left\vert
Dd\left( t_{0},x_{0}\right)  \right\vert ^{2}\right)  .
\]
Since necessarily $v\left(  \eta\left(  t\right)  \right) =
v\left( \eta\left(  0\right)  \right)  =0$ for $t\geq0$
sufficiently small, (\ref{eqsps}) and (\ref{min}) imply that
$\lambda\geq0$.

Now, (\ref{phiissps}) trivially holds if $\lambda=0$. Suppose
$\lambda>0$,
then using (\ref{grad phi}), (\ref{hess phi}) we deduce%
\begin{align*}
-F\left(  D\varphi\left(  t_{0},x_{0}\right)  ,D^{2}\varphi\left(  t_{0}%
,x_{0}\right)  \right)   &  \leq-F\left(  \lambda Dd\left(  t_{0}%
,x_{0}\right)  ,\lambda D^{2}d\left(  t_{0},x_{0}\right)  \right)  \\
&  =-\lambda F\left(  Dd\left(  t_{0},x_{0}\right)  ,D^{2}d\left(  t_{0}%
,x_{0}\right)  \right)  ,
\end{align*}
and%
\[
\varphi_{t}\left(  t_{0},x_{0}\right)  =\lambda d_{t}\left(  t_{0}%
,x_{0}\right)  =-\lambda F\left(  Dd\left(  t_{0},x_{0}\right)
,D^{2}d\left( t_{0},x_{0}\right)  \right)  .
\]
Therefore, we conclude that $\varphi$ satisfies (\ref{phiissps})
at $\left( t_{0},x_{0}\right)  $.

A completely analogous proof shows that $v_{\varepsilon }$ is a
supersolution, and $w$ and $w_{\varepsilon}$ are subsolutions to
(\ref{curvatureevolutioneq}).
\end{proof}

\medskip

\begin{rem}
{\em In the very special case of the evolution by mean curvature
in arbitrary codimension (see Example \ref{mean curvature in
arbitrary codimension}) the above proof breaks down because the
sets $\Gamma_t$ are no longer hypersurfaces of $M$, but
$k$-codimensional submanifolds. The consistency of the generalized
motion (given, e.g., by Corollary \ref{generalized motion by mean
curvature in arbitrary codimension}) with the classical evolution
thus remains an open problem (whose solution would probably
require a careful analysis of the properties of the distance
functions to the submanifolds $\Gamma_t$, and of the eigenvalues
of their Hessians, similar to the study carried out in \cite{AS}
for the case $M=\mathbb{R}^{n}$).}
\end{rem}

\bigskip

\section{Counterexamples and conjectures}

In this final section we provide some counterexamples showing that
many well known properties of the evolutions by mean curvature in
$\mathbb{R}^{n}$ fail when $M$ is a Riemannian manifold of negative
curvature.

\begin{ex}\label{d is not supersolution when curvature is negative}
{\em When $\left(  M,g\right)  $ is the Euclidean space equipped
with the canonical metric, Ambrosio and Soner have proved in
\cite{AS} that the distance function $\left\vert d\right\vert $ is
always a supersolution to the mean curvature equation
\eqref{meancurvature}. This is no longer the case of a general
Riemannian manifold, as the following example shows.

Let $\left(  M,g\right)  $ be a surface of revolution embedded in
$\mathbb{R}^{3}$, locally parameterized by:%
\[
I\times\left(  -\pi,\pi\right)  \ni\left(  s,\theta\right)
\mapsto\left( r\left(  s\right)  \cos\theta,r\left(  s\right)
\sin\theta,s\right) \in\mathbb{R}^{3},
\]
where $I\subseteq\mathbb{R}$ is an open interval and $r\in
C^{\infty}\left( I\right)  $ with $r\geq\rho>0$. In these
coordinates, the metric takes the
form:%
\[
\left(
\begin{array}
[c]{cc}%
r^{\prime}\left(  s\right)  ^{2}+1 & 0\\
0 & r\left(  s\right)  ^{2}%
\end{array}
\right).
\]
Suppose $0\in I$ and $r^{\prime}\left(  0\right)  =0$; then the
\textquotedblleft Equator\textquotedblright\ $s=0$ is a geodesic of
$M$, denote it by $\Gamma_{0}$. The classical evolution by mean
curvature starting from $\Gamma_{0}$ is constant in time. Therefore,
the corresponding signed distances satisfy $d\left( t,\cdot\right)
=d\left(  0,\cdot\right)  \equiv d$ for every $t\in\mathbb{R}$ (we
shall assume $d>0$ for $s>0$). Let us next explicitly compute $d$.
The geodesics of $M$ that are orthogonal to $\Gamma_{0}$ are of the
form $\left(  s\left(  t\right) ,\theta\right)  $ with
$\theta\in\left(  -\pi,\pi\right)  $ constant. Take such a geodesic
and assume that is parameterized by arc length and satisfies
$s^{\prime}>0$. In particular, since its tangent vector $\left(
s^{\prime},0\right)  $ must be of norm one,
\begin{equation}
v\left(  s\left(  t\right)  \right)  ^{2}\left(  s^{\prime}\left(
t\right) \right)  ^{2}=1,\qquad v\left(  s\right)
:=\sqrt{r^{\prime}\left(  s\right)
^{2}+1}.\label{n1}%
\end{equation}
Clearly, $d\left(  s\left(  t\right)  ,\theta\right)  =t$ and
$\partial _{s}d\left(  s\left(  t\right)  ,\theta\right)
s^{\prime}\left(  t\right) =1$. Identity (\ref{n1}) and our
assumption $s^{\prime}>0$ allow us to conclude that
$\partial_{s}d\left(  s,\theta\right)  =v\left(  s\right)  $ and
$Dd=\left(  1/v,0\right)  $. Finally, a direct computation gives:%
\[
\left\vert Dd\right\vert \operatorname*{div}\left(
\frac{Dd}{\left\vert Dd\right\vert }\right)  =\frac{r^{\prime}}{rv}.
\]
The function $\left\vert d\right\vert $ will be a supersolution to
the mean
curvature equation provided $r^{\prime}\left(  s\right)  \operatorname{sign}%
\left(  s\right)  \leq0$. This is always the case if the curvature
of $M$ remains nonnegative. On the other hand, $\left\vert
d\right\vert $ will a subsolution whenever the curvature of $M$ is
nonpositive everywhere. Finally, taking for instance $r\left(
s\right)  =1+\cos^{2}s  $ we are able to produce a $\left\vert
d\right\vert $ that is not a supersolution, neither a subsolution to
the mean curvature equation.

\medskip

\noindent {\bf Conjecture:} If $M$ has nonnegative sectional
curvature then $|d|$ is always a supersolution. If $M$ has
negative curvature then there always exists $\Gamma_{0}$ such that
$|d|$ is not a supersolution.}
\end{ex}

\medskip

\begin{ex}\label{increasing of the distance fails with negative
curvature} {\em When $M=\mathbb{R}^{n}$, Evans and Spruck
\cite[Theorem 7.3]{ES1} showed that if $\Gamma_{0}$,
$\hat{\Gamma}_{0}$ are compact level sets and $\Gamma_{t}$,
$\hat{\Gamma}_{t}$ are the corresponding generalized evolutions by
mean curvature, then
$$
\textrm{dist}(\Gamma_{0}, \hat{\Gamma}_{0})\leq \textrm{dist}(\Gamma_{t}, \hat{\Gamma}_{t})
$$
for all $t>0$. This result fails for manifolds of negative
curvature. For instance, let $M=\{(x,y,z)\in\mathbb{R}^{3}:
x^{2}+y^{2}=1+z^{2}\}$ be a hyperboloid of revolution embedded in
$\mathbb{R}^{3}$. Let $$\Gamma_{0}=\{(x,y,z)\in M: z=0\},$$ and
$$\hat{\Gamma}_{0}=\{(x,y,z)\in M: z=1\}.$$ Then
$$\Gamma_{t}=\Gamma_{0} \textrm{ for all } t>0, \qquad \textrm{and}
\qquad \textrm{dist}(\Gamma_{0}, \hat{\Gamma}_{0})>0,
$$
but $$\textrm{dist}({\Gamma}_{t}, \hat{\Gamma}_{t})=\textrm{dist}(
{\Gamma}_{0}, \hat{\Gamma}_{t})\to 0 \textrm{ as } t\to \infty.$$

\medskip

\noindent {\bf Conjecture:} Evans-Spruck's \cite[Theorem 7.3]{ES1}
result holds true for all manifolds of nonnegative sectional
curvature, but fails for all manifolds of negative curvature. }
\end{ex}

\medskip

\begin{ex}\label{no conservation of Lip constant for negative curvature}
{\em In the case $M=\mathbb{R}^{n}$ it is well known that equation
\eqref{curvatureevolutioneq} preserves Lipschitz properties of the
initial data. Namely, if $g$ is $L$-Lipschitz and $u$ is the unique
solution of \eqref{curvatureevolutioneq} then $u(t, \cdot)$ is
$L$-Lipschitz too, for all $t>0$; see \cite[Chapter 3]{Giga}. Since
the proof of Theorem 7.3 in \cite{ES1} remains valid for any
manifold provided that one assumes the Lipschitz preserving property
of \eqref{curvatureevolutioneq}, the preceding example also shows
that \eqref{curvatureevolutioneq} does not preserve Lipschitz
constants when $M$ is a hyperboloid of revolution.

\medskip

\noindent {\bf Conjecture:} The equation
\eqref{curvatureevolutioneq} has the Lipschitz preserving property
if and only if $M$ has nonnegative sectional curvature. }
\end{ex}


\bigskip

\section{Appendix: Existence and uniqueness of viscosity solutions to a (nonsingular) general parabolic equation}

In this appendix we present the standard definition of viscosity
solution and state existence and comparison result for viscosity
solutions to non-singular parabolic fully nonlinear equations.

 \begin{defn}
{\em Let $M$ be a finite-dimensional Riemannian manifold, and a
continuous function $F:(0,T) \times M \times \mathbb{R}\times
J^{2}M\to\mathbb{R}$. Consider the parabolic equation
 \begin{equation}\label{generalequation}
u_t+F(t,x,u,Du,D^2u)=0, \end{equation} where $u$ is a function of
$(t,x)$. We say that an USC function $u:(0,T)\times
M\to\mathbb{R}$ is a viscosity subsolution of the partial
differential evolution equation provided that
    $
    a+F(t,x, u(t,x), \zeta, A)\leq 0
    $
for all $(t,x)\in (0,T)\times M$ and $(a,\zeta, A)\in
\mathcal{P}^{2,+}u(t,x)$. Similarly, a viscosity supersolution of
\eqref{generalequation} is a LSC function $u:(0,T)\times
M\to\mathbb{R}$ such that
    $
   a+ F(t,x, u(t,x), \zeta, A)\geq 0
    $
for every $(t,x)\in (0,T)\times M$ and $(a,\zeta, A)\in
\mathcal{P}^{2,-}u(t,x)$. If $u$ is both a viscosity subsolution
and a viscosity supersolution of  $u_t+F(t,x,u,Du,D^2u)=0$, we say
that $u$ is a viscosity solution.}
\end{defn}

\begin{rem}
{\em If $u$ is a subsolution of  $u_t+F(t,x,u,Du,D^2u)=0$ and $F$
is lower semicontinuous, then $a+F(t,x, u(t,x), \zeta, A)\leq 0$
for every $(a,\zeta, A)\in \mathcal{\overline{P}}^{\, 2,+}u(t,x)$
and every $(t,x) \in (0,T)\times M$. A similar observation applies
to supersolutions when $F$ is upper semicontinuous, and to
solutions when $F$ is continuous.}
\end{rem}

 \begin{thm} \label{general}
Let $M$ be a compact and $F:(0,T)\times M\times\mathbb{R}\times
J^{2}M\rightarrow\mathbb{R}$ be continuous, proper, and such that

\begin{enumerate}

\item  there exists $\gamma>0$ with
    $$\gamma(r-s)\leq F(t,x,r,\zeta,Q)-F(t,x,s,\zeta,Q)$$
for all $r\geq s$, $(t,x,\zeta, Q)\in (0,T)\times M\times
J^{2}M)$;
\item there exists a function $\omega:[0,\infty]\rightarrow[0,\infty]$
with $\lim_{t\to 0^{+}}\omega(t)=0$ and such that, for every
$\alpha>0$,
    $$
    F(t,y,r, \alpha\exp_{y}^{-1}(x),Q)-F(t,x,r,-\alpha\exp_{x}^{-1}(y),P)
    \leq \omega\left(\alpha d(x,y)^2+d(x,y)\right)
    $$
for all $t\in (0,T)$, $x,y\in M$, $r\in\mathbb R$, $P\in
\mathcal{L}^2_{s}(TM_x), Q\in \mathcal{L}_{s}^2(TM_y)$ with
\begin{equation*}
        -\left(
        \frac{1}{\varepsilon_{\alpha}}+\|A_{\alpha}\|\right)I\leq
\left(\begin{array}{cc}
 P & 0  \\
 0 & -Q \\
\end{array}\right)\leq A_{\alpha}+\varepsilon_{\alpha} A_{\alpha}^{2}
\end{equation*}
and
$\displaystyle{\varepsilon_{\alpha}=(2+2\|A_{\alpha}\|)^{-1}}$,
where $A_{\alpha}$ is the second derivative of the function
$\varphi_{\alpha}(x,y)=\frac{\alpha}{2}d(x,y)^{2}$ at the point
$(x,y)\in\ M\times M$ with  $d(x,y)<\min\{ i_{M}(x),
\,i_{M}(y)\}.$
\end{enumerate}
Assume that   $u\in USC([0,T)\times M)$ is a subsolution  and
$v\in LSC([0,T)\times M)$ is a supersolution  of
\eqref{generalequation} on $M$ and $u \le v$ on $\{0\}\times M$.
Then $u \le v$ on $[0,T)\times M$.
\end{thm}

The proof of this result is a combination of the ideas in the
proof of \cite[ Theorem 8.2]{CIL} with the new techniques for
second order nonsmooth analysis on manifolds developed when
dealing with the stationary case in \cite[Theorem 4.2]{AFS}. See
also Theorem \ref{keyparabolic}, which should be used in this
proof. We leave the details to the reader's care.

Analogous results to Corollary 4.6 and Corollary 4.8 stated in
\cite{AFS} can also be obtained in a similar way for these general
evolution equations. We say in such cases that ``comparison
holds".

As it is customary in these cases, whenever the Comparison Theorem
holds there is no difficulty in applying Perron's method (see
\cite{CIL, I}) to show that the problem
\begin{equation} \label{generalequation2} \tag{\texttt{GEE}}
\left\{
  \begin{array}{ll}
    u_{t}+F(t,x, u(t,x), Du(t,x), D^{2}u(t,x))= 0, \\
    u(0,x)=g(x).
  \end{array}
\right.
\end{equation}
has a unique bounded viscosity solution $u$ on $M$, provided that
comparison holds and one is able to find a viscosity subsolution
$\underline{u}$ and a viscosity supersolution $\overline{u}$ such
that $\underline{u}_*(0,x)=\overline{u}^*(0,x)=g(x)$. In fact $u$
is given by
\begin{equation*}
  W(t,x)=\sup\{w(t,x):\underline{u}\leq w\leq \overline{u} \text{ and }
    w \text{ is a subsolution of } \eqref{generalequation}\}.
\end{equation*}

\bigskip

\begin{center}
{\bf Acknowledgements}
\end{center}
We thank Juan Ferrera for several helpful conversations related to
this paper. M. Jimenez-Sevilla is indebted to the Department of
Mathematics at the Ohio State University and very specially to
David Goss for his kind hospitality.

\bigskip



\begin{thebibliography}{}

\bibitem{A}
L. Ambrosio, {\em Geometric evolution problems, distance function
and viscosity solutions.} Calculus of variations and partial
differential equations (Pisa, 1996), 5--93, Springer, Berlin,
2000.

\bibitem{AS}
L. Ambrosio, H. M. Soner, {\em Level set approach to mean
curvature flow in arbitrary codimension.} J. Differential Geom. 43
(1996), no. 4, 693--737.

\bibitem{Andrews}
B. Andrews,{\em Gauss curvature flow: the fate of the rolling
stones.} Invent. Math. 138 (1999), no. 1, 151--161.

\bibitem{AFL1} D. Azagra, J. Ferrera, F. L{\'o}pez-Mesas, {\em
Nonsmooth analysis and Hamilton-Jacobi equations on Riemannian
manifolds}, J. Funct. Anal. 220 (2005) no. 2, 304-361.


\bibitem{AFS} D. Azagra, J. Ferrera, B. Sanz, {\em Viscosity
solutions to second order partial differential equations I},
preprint, 2006 (math.AP/0612742v1).

\bibitem{Brakke}
K.A. Brakke, {\em The motion of a surface by its mean curvature.}
Mathematical Notes, 20. Princeton University Press, Princeton, N.J.,
1978.

\bibitem{ChGG}
Y. G. Chen, Y. Giga, S. Goto, {\em Uniqueness and existence of
viscosity solutions of generalized mean curvature flow equations.}
J. Differential Geom. 33 (1991), no. 3, 749--786.

\bibitem{dC}
M.P. do Carmo, {\em Riemannian Geometry}, Mathematics: Theory \&
Applications, Birkh{\"a}user Boston, 1992.

\bibitem{CIL} M.G. Crandall, H. Ishii, P.-L. Lions, {\em User's
guide to viscosity solutions of second order partial differential
equations}, Bull. Amer. Math. Soc. 27 (1992) no. 1, 1-67.

\bibitem{DaviniSino}
A. Davini, A. Siconolfi, {\em A generalized dynamical approach
to the large time behavior of solutions of Hamilton-Jacobi equations.}
SIAM J. Math. Anal. 38 (2006), no. 2, 478--502.

\bibitem{FathiSino1}
A. Fathi, A. Siconolfi, {\em Existence of $C\sp 1$
critical subsolutions of the Hamilton-Jacobi equation.} Invent. Math.
155 (2004), no. 2, 363--388.

\bibitem{FathiSino2}
A. Fathi, A. Siconolfi, {\em PDE aspects of Aubry-Mather theory
for quasiconvex Hamiltonians.} Calc. Var. Partial Differential Equations
22 (2005), no. 2, 185--228.

\bibitem{ES1}
L. C. Evans, J. Spruck, {\em Motion of level sets by mean curvature.
I.} J. Differential Geom. 33 (1991), no. 3, 635--681.

\bibitem{ES2}
L. C. Evans, J. Spruck, {\em Motion of level sets by mean curvature.
II.} Trans. Amer. Math. Soc. 330 (1992), no. 1, 321--332.

\bibitem{Firey}
W.J. Firey, {\em Shapes of worn stones.} Mathematika 21 (1974),
1--11.

\bibitem{GageHamilton}
M. Gage, R.S. Hamilton, {\em The heat equation shrinking convex
plane curves.} J. Differential Geom. 23 (1986), no. 1, 69--96.


\bibitem{Giga}
Y. Giga, {\em Surface evolution equations. A level set approach.}
Monographs in Mathematics, 99. Birkh{\"a}user Verlag, Basel, 2006.

\bibitem{GGIS}
Y. Giga, S. Goto, H. Ishii, M.-H. Sato, {\em Comparison principle
and convexity preserving properties for singular degenerate
parabolic equations on unbounded domains.} Indiana Univ. Math. J. 40
(1991), no. 2, 443--470.

\bibitem{Goto}
S. Goto, {\em Generalized motion of hypersurfaces whose growth speed
depends superlinearly on the curvature tensor.} Differential
Integral Equations 7 (1994), no. 2, 323--343.

\bibitem{Grayson}
M.A. Grayson, {\em A short note on the evolution of a surface by its
mean curvature.} Duke Math. J. 58 (1989), no. 3, 555--558.

\bibitem{Grayson2}
M.A. Grayson, {\em The heat equation shrinks embedded plane curves
to round points.} J. Differential Geom. 26 (1987), no. 2, 285--314.

\bibitem{Grayson3}
M.A. Grayson, {\em Shortening embedded curves.} Ann. of Math. (2)
129 (1989), no. 1, 71--111.

\bibitem{GurskyViaclovsky}
M. Gursky and J. Viaclovsky, {\em Prescribing symmetric functions
of the eigenvalues of the Ricci tensor}, to appear in Annals of
Math.

\bibitem{Huisken1}
G. Huisken, {\em Flow by mean curvature of convex surfaces into
spheres.} J. Differential Geom. 20 (1984), no. 1, 237--266.

\bibitem{Huisken2}
G. Huisken, {\em Contracting convex hypersurfaces in Riemannian
manifolds by their mean curvature.} Invent. Math. 84 (1986), no. 3,
463--480.

\bibitem{Ilmanen}
T. Ilmanen, {\em Generalized flow of sets by mean curvature on a
manifold.} Indiana Univ. Math. J. 41 (1992), no. 3, 671--705.

\bibitem{Ilmanen2}
T. Ilmanen, {\em Elliptic regularization and partial regularity for
motion by mean curvature.} Mem. Amer. Math. Soc. 108 (1994), no.
520.

\bibitem{I}
H. Ishii, {\em Perron's method for Hamilton-Jacobi equations}, Duke
Math. J. 55 (1987), no. 2, 369--384.

\bibitem{IM}
H. Ishii, T. Mikami, {\em A level set approach to the wearing
process of a nonconvex stone.} Calc. Var. Partial Differential
Equations 19 (2004), no. 1, 53--93.

\bibitem{IS}
H. Ishii,  P. Souganidis, {\em Generalized motion of noncompact
hypersurfaces with velocity having arbitrary growth on the curvature
tensor}, Tohoku Math. J. (2) 47 (1995), no. 2, 227--250.

\bibitem{LedZhu}
Y. S. Ledyaev and Q.J. Zhu, {\em Nonsmooth analysis on smooth
manifolds}, to appear in Transactions of the AMS.


\bibitem{ManteMenu}
C. Mantegazza and A.C. Mennucci,  {\em Hamilton-Jacobi equations
and distance functions on Riemannian manifolds}, Appl. Math.
Optim. 47 (2003), no. 1, 1--25.

\bibitem{Sakai}
T. Sakai, {\em Riemannian Geometry}, Translations of Mathematical
Monographs, vol 149, Amer. Math. Soc. 1992

\bibitem{Soner}
H.M. Soner, {\em Motion of a set by the curvature of its boundary.}
J. Differential Equations 101 (1993), no. 2, 313--372.

\bibitem{Soner2}
H.M. Soner, {\em Front propagation. Boundaries, interfaces, and transitions}
(Banff, AB, 1995), 185--206, CRM Proc. Lecture Notes, 13, Amer. Math. Soc., Providence, RI, 1998.

\bibitem{Spivak}
M. Spivak, {\em  A comprehensive introduction to differential
geometry, Vol. IV}. Second edition. Publish or Perish, Inc.,
Wilmington, Del., 1979.

\bibitem{Tso}
K. Tso, Kaising, {\em Deforming a hypersurface by its
Gauss-Kronecker curvature.} Comm. Pure Appl. Math. 38 (1985), no. 6,
867--882.


\end{thebibliography}
\end{document}